\documentclass[11pt, reqno]{amsart}
\usepackage{amsmath, amssymb, mathrsfs, amsthm}
\usepackage{bm}
\usepackage{graphicx}
\usepackage{tikz} 
\usetikzlibrary{arrows, shapes, chains} 
\usepackage{booktabs}
\usepackage{float}
\usepackage{subfig}
\usepackage[colorlinks, linkcolor=blue, anchorcolor=blue, citecolor=blue, CJKbookmarks=True]{hyperref}
\usepackage{geometry}
\usepackage{color, xcolor}

\newtheorem{assumption}{Assumption}
\newtheorem{definition}{Definition}[section]
\newtheorem{lemma}[definition]{Lemma}
\newtheorem{theorem}[definition]{Theorem}
\newtheorem{proposition}[definition]{Proposition}

\numberwithin{equation}{section}
\allowdisplaybreaks[4]

\newcommand{ \rd}{\mathrm d}
\newcommand{\hC}{\mathbb C}
\newcommand{\hE}{\mathbb E}
\newcommand{\hN}{\mathbb N}
\newcommand{\hP}{\mathbb P}
\newcommand{\hR}{\mathbb R}

\newcommand{\cB}{\mathcal B}
\newcommand{\cC}{\mathcal C}
\newcommand{ \cE}{\mathcal E}
\newcommand{\cF}{\mathcal F}
\newcommand{\cG}{\mathcal G}

\newcommand{\cJ}{\mathcal J}
\newcommand{\cK}{\mathcal K}
\newcommand{\cM}{\mathcal M}

\newcommand{\cR}{\mathcal R}
\newcommand{\cY}{\mathcal Y}

\newcommand{\<}{\langle}
\renewcommand{\>}{\rangle}

\begin{document}

\title[Error distribution of MLE method for FSDE]{Asymptotic error distribution of Mittag--Leffler Euler method for a fractional stochastic differential equation}

\author{Xinjie Dai}
\address{School of Mathematics and Statistics, Yunnan University, Kunming 650500, Yunnan, China}
\email{dxj@ynu.edu.cn}

\author{Baiping Zhang}
\address{School of Mathematics and Statistics, Yunnan University, Kunming 650500, Yunnan, China}
\email{zhangbaiping@stu.ynu.edu.cn}

\author{Diancong Jin}
\address{School of Mathematics and Statistics, Huazhong University of Science and Technology, Wuhan 430074, China; 
Hubei Key Laboratory of Engineering Modeling and Scientific Computing, Huazhong University of Science and Technology, Wuhan 430074, China}
\email{jindc@hust.edu.cn (Corresponding author)}

\thanks{This work is supported by National Natural Science Foundation of China (Nos.\ 12401547, 12201228, 12471391), and Yunnan Fundamental Research Project (No.\ 202501AU070074), and Scientific Research and Innovation Project of Postgraduate Students in the Academic Degree of Yunnan University (No.\ KC-252513127)}

\begin{abstract}
In this paper, we investigate the asymptotic distribution of the normalized error for the Mittag–Leffler Euler (MLE) method applied to a class of multidimensional fractional stochastic differential equations. These equations are reformulated as stochastic Volterra equations (SVEs) featuring a non-diagonal, matrix-valued kernel $K(u)=u^{\alpha-1}E_{\alpha,\alpha}(Au^{\alpha})$ with singular exponent $\alpha \in (\frac{1}{2}, 1)$. To enhance computational efficiency, the singular kernel is discretized using the left-rectangle rule, posing technical challenges for the theoretical analysis. To address this, we introduce an auxiliary $K$-undiscretized scheme to bridge the gap between the exact solution and the MLE method, integrating Jacod's stable convergence theory for conditional Gaussian martingales with methodologies developed for SVEs. To the best of our knowledge, this is the first work to establish the asymptotic error distribution for numerical methods incorporating non-diagonal matrix-valued kernels. 
\end{abstract}


\maketitle


\section{Introduction}
 \label{sec.Intro}

Consider the $d$-dimensional semilinear fractional stochastic differential equation (FSDE)
\begin{align} \label{eq.FSDE}
D_{c}^{\alpha} X_t = A X_t + b(X_t) + \sigma(X_t) \dot{W}_t, \qquad t \in [0, T], 
\end{align} 
where $D_{c}^{\alpha}$ denotes the Caputo fractional derivative of order $\alpha \in (\frac{1}{2}, 1)$, $A \in \hR^{d \times d}$, $b: \hR^d \rightarrow \hR^d$, $\sigma: \hR^d \rightarrow \hR^{d \times m}$, and $W$ is an $m$-dimensional Brownian motion defined on some complete filtered probability space $(\Omega, \cF, \{\cF_t\}_{0 \leq t \leq T}, \mathbb{P})$ satisfying the usual conditions. Typical applications of \eqref{eq.FSDE} arise from the spatial discretizations of fractional time stochastic partial differential equations (SPDEs), rough volatility models in mathematical finance, and anomalous diffusion models in statistical physics; see, e.g., \cite{AbiLarsson2019, BayerBook, ChenKim2015, LiLiuLu2017} and references therein.

As in classical stochastic differential equations (SDEs, corresponding to the case $\alpha = 1$), the differential form \eqref{eq.FSDE} is purely symbolic and should be understood as the integral equation
\begin{align} \label{eq.SIE}
X_{t} = X_{0} + \frac{1}{\Gamma(\alpha)} \int_{0}^{t} (t-s)^{\alpha-1} \big( A X_s + b(X_s) \big) \rd s + \frac{1}{\Gamma(\alpha)} \int_{0}^{t} (t-s)^{\alpha-1} \sigma(X_s) \rd W_{s}, 
\end{align}
where $\Gamma$ denotes the Gamma function. If the drift coefficient $b$ and the diffusion coefficient $\sigma$ together satisfy the global Lipschitz condition, then the equation \eqref{eq.SIE} admits a unique strong solution; see, e.g., \cite{AbiLarsson2019, DoanKloedenHuongTuan2018}. Alternatively, the semilinear FSDE \eqref{eq.FSDE} can be interpreted as the following stochastic Volterra equation (SVE) in convolution form:\ 
\begin{align} \label{eq.SVE}
X_{t} = E_{\alpha}(A t^{\alpha}) X_{0} + \int_{0}^{t} K(t-s) b(X_s) \rd s + \int_{0}^{t} K(t-s) \sigma(X_s) \rd W_{s}, 
\end{align}
based on the equivalence between the strong and mild solutions of \eqref{eq.FSDE}; see \cite{AnhDoanHuong2019, Tuan2021}. Here, the $\hR^{d \times d}$-valued Volterra kernel $K(u) := u^{\alpha-1} E_{\alpha, \alpha}(A u^{\alpha})$ is singular at the origin for $\alpha \in (\frac{1}{2}, 1)$, and $E_{\alpha}$ and $E_{\alpha, \alpha}$ denote the Mittag--Leffler functions; see Appendix \ref{appen:ML function} for more details.

It is well known that the error analysis of direct discretizations for the Caputo derivative $D_{c}^{\alpha} X_t$ relies, often as a standard assumption, on the time differentiability of the solution $X_t$ \cite{LinXu2007, SunWu2006}, a property that fails in the stochastic context due to the almost surely non-differentiability of the Brownian motion $W$. Perhaps for this reason, existing numerical methods for the FSDE \eqref{eq.FSDE} are constructed via approximating its integral form \eqref{eq.SIE} or \eqref{eq.SVE}. On one hand, numerical approximations of the equation \eqref{eq.SIE} include the Euler--Maruyama (EM) method, the fast EM method, the $\theta$-EM method, and the Milstein method; see, e.g., \cite{AlfonsiKebaier2024, DaiXiao2020, FangLi2020, JourdainPages2025, LiHuangHu2022, RichardTanYang2021, Zhang2008} and references therein. On the other hand, \cite{DoanHuongKloedenVu2020} proposed an exponential EM method, which we refer to in this paper as the Mittag--Leffler Euler (MLE) method, to solve the equation \eqref{eq.SVE}. More specifically, the MLE method applied to the semilinear FSDE \eqref{eq.FSDE} is formulated as follows:\ 
\begin{align} \label{eq.MLE}
\widehat{X}_{t} = E_{\alpha}(A t^{\alpha}) X_{0} + \int_{0}^{t} K(t-\underline{s}) b(\widehat{X}_{\underline{s}}) \rd s + \int_{0}^{t} K(t-\underline{s}) \sigma(\widehat{X}_{\underline{s}}) \rd W_{s}, \qquad t \in [0, T], 
\end{align} 
where $\underline{s} := \left \lfloor s/h \right \rfloor\! h$ denotes the left grid point of $s \in [0, T]$ with step size $h = T/n$. As a main advantage over other numerical methods, the MLE method \eqref{eq.MLE} unconditionally inherits the asymptotic mean-square stability of the scalar bilinear test equation \cite{DoanHuongKloedenVu2020}.

At present, strong approximation methods for FSDEs have been extensively investigated, and these advances have in turn stimulated further research on the associated asymptotic error distribution. The study of the asymptotic error distribution originated in \cite{KurtzProtter1991} and has since developed rapidly, with an expanding literature and substantial progress for numerical schemes of SDEs \cite{HLS2026,Hu2016, Jacod1997, JacodProtter1988, Jin2025arXiv, Liu2023, Protter2020}, SPDEs \cite{Hong2025arXiv, Hong2024arXiv}, and SVEs \cite{FukasawaHojo2025, FukasawaUgai2023, LiuHuGao2025, NualartSaikia2023}. The asymptotic error distribution of a stochastic numerical method refers to the limiting distribution of the normalized error between the exact solution and its numerical approximation as the step size tends to zero, where the normalization is determined by the strong convergence order of the method. For example, the normalized error of the MLE method \eqref{eq.MLE} is given by 
\begin{align*} 
U_t^n := n^{\alpha-\frac{1}{2}} (X_t - \widehat{X}_{t}), \qquad t \in [0,T], 
\end{align*} 
where the strong convergence order $\alpha-\frac{1}{2}$ follows from a slight extension of the result in \cite{DoanHuongKloedenVu2020}. Then, the asymptotic error distribution of the MLE method \eqref{eq.MLE} is defined as the limiting distribution of $U_t^n$ as the step size $h$ tends to zero, or equivalently, as the total number of time steps $n$ goes to infinity. As a kind of generalized central limit theorem, once the asymptotic error distribution of the MLE method \eqref{eq.MLE} is established, the resulting nontrivial limiting distribution not only confirms the optimality of the strong convergence order $\alpha-\frac{1}{2}$ but also determines the optimal choice of tuning parameters for the related multilevel Monte Carlo method \cite{BenAlaya2015, Liu2025}. Together with the aforementioned numerical stability, this motivates us to study the asymptotic error distribution of the MLE method \eqref{eq.MLE} in this paper.

For the MLE method \eqref{eq.MLE}, there exists a variant in which the kernel is not discretized, given by 
\begin{align} \label{eq.KMLE}
\bar{X}_{t} = E_{\alpha}(A t^{\alpha}) X_{0} + \int_{0}^{t} K(t-s) b(\bar{X}_{\underline{s}}) \rd s + \int_{0}^{t} K(t-s) \sigma(\bar{X}_{\underline{s}}) \rd W_{s}, 
\end{align}
whose strong convergence order is also $\alpha - \frac{1}{2}$. The variant \eqref{eq.KMLE} requires the simulation of a large Gaussian covariance matrix, whereas the MLE method \eqref{eq.MLE} requires only the simulation of independent Brownian increments. When the matrix $A$ reduces to the constant $0$, \cite{FukasawaUgai2023} shows that the normalized error of the variant \eqref{eq.KMLE} stably converges in law in the H\"older space $\cC_0^{\alpha-\frac{1}{2}-\varepsilon}$ to the solution of a certain linear SVE, based on Jacod's theory of convergence in distribution for conditional Gaussian martingales. However, for the $K$-discrete MLE method \eqref{eq.MLE}, the limiting distribution of the normalized error cannot be expected to possess any H\"older continuity properties; see \cite{NualartSaikia2023} for the case where the matrix $A = 0 \in \hR$ and the drift coefficient $b \equiv 0 \in \hR^d$. Therefore, to analyze the asymptotic error distribution of the low-cost MLE method \eqref{eq.MLE}, we introduce an auxiliary $K$-undiscretized scheme 
\begin{align} \label{eq.auxiMLE}
\widetilde{X}_{t} = E_{\alpha}(A t^{\alpha}) X_{0} + \int_{0}^{t} K(t-s) b(\widehat{X}_{\underline{s}}) \rd s + \int_{0}^{t} K(t-s) \sigma(\widehat{X}_{\underline{s}}) \rd W_{s} 
\end{align}
and decompose the normalized error $U_t^n$ as follows:\ 
\begin{align*} 
U_t^{n} 
= \underbrace{ n^{\alpha-\frac{1}{2}} (X_t-\widetilde{X}_t) }_{=:\, Y_t^{n}} + \underbrace{ n^{\alpha-\frac{1}{2}} (\widetilde{X}_t -\widehat{X}_t) }_{=:\, R_t^{n}}, \qquad t \in [0,T]. 
\end{align*} 
By leveraging Jacod's stable convergence theory for conditional Gaussian martingales and extending the analytical framework for SVEs, we establish that $Y^n$ converges stably in law to the unique solution of a linear SVE and that the finite-dimensional distributions of $R^n$ converge to a centered Gaussian process, as detailed in Theorems \ref{thm.Yn} and \ref{thm.An}, respectively. These results provide a rigorous theoretical foundation for the optimality of the strong convergence order $\alpha-\frac{1}{2}$ for the MLE method \eqref{eq.MLE}, thereby justifying its efficiency in the simulation of FSDEs. Beyond the present scope, our approach possesses potential applications in characterizing the asymptotic error distributions of numerical methods for fractional SPDEs. We would also mention that the properties of Volterra kernels play a crucial role in the study of asymptotic error distributions, and to the best of our knowledge, the case of a general matrix kernel $K$ with $A \in \hR^{d \times d}$ is investigated for the first time in this paper. We also refer to \cite{FukasawaHojo2025} for the case of diagonal matrix kernels.

The remainder of the paper is organized as follows. In Section \ref{sec.MainRes}, we present the main results concerning the limiting distribution of the normalized error. Section \ref{sec.ProofsMainRes} is devoted to the proofs of main results. Sections \ref{sec.asym_behav_Y_n} and \ref{sec.asym_behav_R_n} provide auxiliary lemmas on the asymptotic behaviors of $Y^n$ and $R^n$, respectively. Finally, Appendix \ref{appen:ML function} recalls the definition of the Mittag--Leffler functions and collects useful estimates for the kernel function $K$.

\textbf{Notation}. Throughout this paper, we use the following notation. 
\begin{itemize}
\item $\hE $: The expectation corresponding to $\hP$.
\vskip 0.3em

\item $\lfloor \,\cdot\, \rfloor$ and $\lceil \,\cdot\, \rceil$: the floor and ceil functions, respectively.
\vskip 0.3em

\item For $u>0$, we may write the kernel function as
\[
K(u) = u^{\alpha-1} E_{\alpha, \alpha}(A u^{\alpha}) = \cK(u)\,\cE(u),
\]
where $\cK(u) = u^{\alpha-1}$ and $\cE(u) = E_{\alpha, \alpha}(A u^{\alpha})$; see Appendix~\ref{appen:ML function} for their properties. 
\vskip 0.3em

\item Let $|\cdot|$ denote both the Euclidean norm on $\mathbb{R}^d$ and the trace (or Frobenius) norm on $\mathbb{R}^{d \times m}$.
\vskip 0.3em

\item For any matrix $M$, we denote by $M^\top$ its transpose, and by $M_{j}^{i}$ its $(i,j)$-th entry. 
\vskip 0.3em

\item $\mathcal{C}_0^{\lambda}$: The space of $\mathbb{R}^d$-valued $\lambda$-H\"older continuous functions on $[0, T]$ vanishing at $t=0$.
\vskip 0.3em

\item $\mathcal{D}_d$: The space of $\mathbb{R}^d$-valued c\`adl\`ag functions on $[0, T]$ equipped with the Skorokhod topology. 

\vskip 0.3em

\item Let $C > 0$ denote a generic positive constant, which may vary from line to line. 
\end{itemize}

\section{Main results}
 \label{sec.MainRes}

To present the main results of this paper, we always assume the matrix $A \in \hR^{d \times d}$ is negative definite. Without loss of generality, we also set $T = 1$. Furthermore, we impose the following assumption on the coefficients of FSDE \eqref{eq.FSDE}.

\begin{assumption} \label{ass.main}
Assume that the drift coefficient $b: \hR^d \to \hR^d$ and the diffusion coefficient $\sigma: \hR^d \to \hR^{d \times m}$ are continuously differentiable, with their derivatives being bounded and uniformly continuous. 
\end{assumption}

Clearly, it follows from Assumption \ref{ass.main} that $b$ and $\sigma$ satisfy the global Lipschitz condition, that is to say, there exists \( L > 0 \) such that for all \( x, y \in \hR^n \),
\begin{align}
|b(x)-b(y)| + |\sigma(x)-\sigma(y)| \leq L|x-y|. \label{eq.lips} 
\end{align} 
Next, we collect some basic properties of the exact and numerical solutions. Their proofs are based on standard arguments using H\"older's inequality, the Burkholder--Davis--Gundy inequality, \eqref{eq.lips}, and Gronwall's inequality. For related proofs, we also refer to \cite{DoanHuongKloedenVu2020, LiHuangHu2022}.

\begin{proposition} \label{pro.X}
Let $X$ be the strong (or mild) solution of \eqref{eq.FSDE}, $\widehat{X}$ and $\widetilde{X}$ be the numerical solutions defined by \eqref{eq.MLE} and \eqref{eq.auxiMLE}, respectively. Then for any $p \geq 1$, there exists $C>0$ such that for any $s,\,t\in [0,T]$, 
\begin{align}
& \| X_t \|_{L^p} + \| \widehat{X}_t \|_{L^p} + \| \widetilde{X}_t \|_{L^p} 
\leq C, \label{eq.momentBounded} \\
& \| X_t - X_s \|_{L^p} + \| \widehat{X}_t - \widehat{X}_s \|_{L^p} + \| \widetilde{X}_t - \widetilde{X}_s \|_{L^p} 
\leq C |t-s|^{\alpha-\frac{1}{2}}, \label{eq.solutionReqularity} \\
& \| X_t - \widehat{X}_t \|_{L^p} + \| X_t - \widetilde{X}_t \|_{L^p} \leq C h^{\alpha-\frac{1}{2}}. \label{eq.strongConvergence}
\end{align}
Here and hereafter, $\| \cdot \|_{L^p} := \| \cdot \|_{L^p (\Omega;\, \hR^d)}$. 
\end{proposition}

As a direct consequence of \eqref{eq.solutionReqularity} and \eqref{eq.strongConvergence}, for any $p \geq 1$, there exists $C>0$ such that for any $s,\,t\in [0,T]$, 
\begin{align} \label{X-hatX_Lp}
\| {X}_t - \widehat{X}_s \|_{L^p} + \| \widetilde{X}_t - \widehat{X}_s \|_{L^p} 
\leq C \big( |t-s|^{\alpha-\frac12} + h^{\alpha-\frac12} \big). 
\end{align}
Moreover, it follows from \eqref{eq.solutionReqularity}, \eqref{eq.strongConvergence}, and the Garsia--Rodemich--Rumsey lemma \cite[Lemma 4.3 or Corollary 4.4]{RichardTanYang2021} that 
\begin{align} \label{X-tildeX_LpC} 
\| X - \widetilde{X} \|_{L^{p}(\Omega;\, \cC([0,T];\, \hR^d))} 
\leq C h^{\alpha-\frac12-\epsilon}, 
\end{align} 
where $\epsilon \in (0,\alpha-\frac{1}{2})$ can be arbitrary small.

Before presenting our main results, we recall the definition of stable convergence in law introduced by \cite{JacodProtter1988}. We remark that stable convergence in law is stronger than convergence in law, yet weaker than convergence in probability.

\begin{definition}Let $\{X_n\}$ be a sequence of random variables taking values in a Polish space $E$, all defined on the same probability space $(\Omega, \mathcal{F}, P)$. We say that $X_n$ converges stably in law to $X$, denoted by $X_n \Rightarrow^{\text{stably}} X$, if there exists an extension
$(\widetilde{\Omega}, \widetilde{\mathcal{F}}, \widetilde{P})$ of the original probability space and if
\begin{align*}
\lim_{n \to \infty} \hE(Hf(X_n)) = \widetilde {\hE}(Hf(X)) 
\end{align*}
for every bounded continuous function $f: E \to \mathbb{R} $ and every bounded measurable random variable $H$. 
\end{definition}

The main results of this paper are summarized in the following two theorems. Recall that the normalized error $U^n$ is decomposed as $U^n = Y^n + R^n$. Our first main result is presented in the following theorem, which characterizes the limiting behavior of $Y^n$.

\begin{theorem} \label{thm.Yn}
As $n$ tends to infinity, the process $Y^{n} = n^{\alpha-\frac{1}{2}} (X-\widetilde{X})$ converges stably in law in $\mathcal{C}_0^{\alpha-\frac{1}{2}-\epsilon}$ to a continuous process $Y = (Y^{1}, \cdots, Y^{d})^{\top}$ for arbitrary $\epsilon \in (0, \alpha - \frac{1}{2})$, which is the unique solution of the linear SVE 
\begin{align} \label{eq.Y_t^i} 
Y_{t}^{i}
&= \sum_{j=1}^{d} \sum_{k=1}^{d} \int_{0}^{t} K^{i}_{j}(t-s) Y_{s}^{k} \Big(\partial_{k} b^{j}(X_{s}) \rd s + \sum_{\ell=1}^{m}\partial_{k} \sigma_{\ell}^{j}(X_{s}) \rd W_{s}^{\ell} \Big) \notag \\ 
&\quad + \kappa_{1}(\alpha) \sum_{j=1}^{d} \sum_{\ell=1}^{m} \sum_{k=1}^{d} \sum_{q=1}^{m} \int_{0}^{t} K^{i}_{j}(t-s) \partial_{k} \sigma_{\ell}^{j}(X_{s}) \sigma_{q}^{k}(X_{s}) \rd B_{s}^{q,\ell}, \qquad i = 1, \cdots, d.
\end{align} 
Here, $B$ is an $m^2$-dimensional standard Brownian motion independent of \(\cF\) and defined on some extension of $(\Omega, \cF, \hP)$, as well as the constant 
\begin{align} \label{eq.kappa1}
\kappa_{1}^{2}(\alpha) 
= \frac{1}{\Gamma^{2}(\alpha)}\left(\int_{0}^{1} \int_{0}^{\infty} \left((y + x)^{\alpha-1}-( \lceil y \rceil)^{\alpha-1} \right)^{2} \rd y \rd x + \frac{1}{2\alpha(2\alpha-1)}\right).
\end{align}
\end{theorem}

To state another main result, let $Z := \{Z_{\ell,t}^{i,j}, t \in [0,T]\}$ denote a $d^2m$-dimensional centered Gaussian process on an extended probability space, independent of $W$, such that 
\begin{align*} 
\hE \big[ Z_{\ell_1,t_1}^{i_1,j_1}Z_{\ell_2,t_2}^{i_2,j_2} \big] 
= 
\begin{cases} 
1, & i_1 = j_1,\, i_2=j_2,\, \ell_1=\ell_2, \text{ and } t_1 = t_2, \\
0, & \text{otherwise}, 
\end{cases}
\end{align*} 
for $(i,i_1,i_2,j,j_1,j_2) \in \{1,2,\cdots,d\}^6$, $(\ell,\ell_1,\ell_2) \in \{1,2,\cdots,m\}^3$, and $(t, t_1, t_2) \in [0,T]^3$. For $t \in [0,T]$, we also define $\widetilde{\cR}_t = (\widetilde{\cR}_t^{1}, \widetilde{\cR}_t^{2}, \cdots, \widetilde{\cR}_t^{d})^{\top}$ with 
\begin{align} \label{eq.def:widetilde_cR}
\widetilde{\cR}_t^{i} 
= \kappa_{2}(\alpha) \sum_{j=1}^{d} \sum_{\ell=1}^{m} \sigma^j_{\ell}(X_t) Z_{\ell,t}^{i,j}, \qquad i \in \{1,2,\cdots,d\}, 
\end{align}
where the constant
\begin{align} \label{eq.kappa2}
\kappa_{2}^{2}(\alpha) = \frac{1}{\Gamma^{2}(\alpha)}\int_{0}^{1}\int_{0}^{\infty}\left((\lceil y \rceil-x)^{\alpha-1} - (\lceil y \rceil)^{\alpha-1} \right)^2 \rd y \rd x. 
\end{align}

The following theorem presents the second main result, which describes the limit behavior of $R^n$.

\begin{theorem} \label{thm.An} 
As $n$ goes to infinity, the finite-dimensional distributions of the process $\{R^{n}_{\underline{t}},t\in(0,T]\}$ converge in distribution to those of $\{\widetilde{\cR}_t,\,t\in(0,T]\}$. 
\end{theorem}




\section{Proofs of main results}
 \label{sec.ProofsMainRes}

This section is devoted to the proofs of Theorems \ref{thm.Yn} and \ref{thm.An}.

\subsection{Proof of Theorem \ref{thm.Yn}}

For the first part $Y^{n}$ of $U^{n}$, it follows from \eqref{eq.SVE} and \eqref{eq.auxiMLE} that for $1 \leq i \leq d$ and $t \in (0,T]$, 
\begin{align}
&\ Y_{t}^{n,i} = n^{\alpha-\frac{1}{2}} (X_{t}^{i} - \widetilde{X}_{t}^{i}) \nonumber \\ 
&= \sum_{j=1}^{d} \left( \int_{0}^{t} K_{j}^{i}(t-s) n^{\alpha-\frac{1}{2}} \big( b^{j}(X_{s}) - b^{j}(\widehat{X}_{\underline{s}}) \big) \rd s + \int_{0}^{t} K_{j}^{i}(t-s) n^{\alpha-\frac{1}{2}} \big( \sigma^{j}(X_{s}) - \sigma^{j}(\widehat{X}_{\underline{s}}) \big) \rd W_{s} \right) \nonumber \\ 
&= \sum_{j=1}^{d} \bigg[ \int_{0}^{t} K_{j}^{i}(t-s) \Big(\nabla b^{j}(\widetilde{X}_{s})^{\top} Y_{s}^{n} \rd s + \sum_{\ell=1}^{m}\nabla \sigma_{\ell}^{j}(\widetilde{X}_{s})^{\top} Y_{s}^{n} \rd W_{s}^{\ell} \Big) \nonumber \\ 
&\quad + \int_{0}^{t} K_{j}^{i}(t-s) n^{\alpha-\frac{1}{2}} \nabla b^{j}(\widetilde{X}_{s})^{\top} (\widetilde{X}_{s}-\widehat{X}_{\underline{s}}) \rd s + \sum_{\ell=1}^{m} \sum_{k=1}^{d} \int_{0}^{t} K_{j}^{i}(t-s) \partial_{k} \sigma_{\ell}^{j}(\widetilde{X}_{s}) \rd \hat{V}_{s}^{n,k,\ell} \bigg] + \hat\Delta_{t}^{n,i}, \label{eq.Y_t^n} 
\end{align} 
where $\widehat{\mathbf{V}}^n = \{ \hat{V}^{n,k,\ell} 
= n^{\alpha-\frac{1}{2}} \int_0^\cdot \widetilde{X}_s^k-\widehat{X}_{\underline{s}}^k \rd W_s^{\ell},\, 1 \leq k \leq d,\, 1 \leq \ell \leq m\}$, and $\widehat{\Delta}^n = (\hat{\Delta}^{n,1}, \ldots, \hat{\Delta}^{n,d})^\top$ is the remainder. To facilitate the proof of Theorem \ref{thm.Yn}, several lemmas are provided below.

\begin{lemma} \label{lm.covV} 
Let $t \in [0, T]$, $(k_1, k_2, k) \in \{1,\ldots,d\}^3$, and $\ell \in \{1,\ldots,m\}$. Then
\begin{align} 
&\< \hat{V}^{n,k_1, \ell}, \hat{V}^{n,k_2, \ell} \>_t
\ \xrightarrow[n \to \infty]{L^1(\Omega)} \ 
\kappa_{1}^{2}(\alpha) \sum_{q=1}^m \int_0^{t} \sigma_{q}^{k_1}(X_s) \sigma_{q}^{k_2}(X_s) \rd s, \label{lm.covV(i)} \\ 
&\< \hat{V}^{n,k,\ell}, W^{\ell} \>_t 
\ \xrightarrow[n \to \infty]{L^1(\Omega)}\ 
0, \label{lm.covV(ii)}
\end{align} 
where $\kappa_{1}^{2}(\alpha)$ is defined in \eqref{eq.kappa1}. 
\end{lemma}

The proof of Lemma \ref{lm.covV} is lengthy and technical, and is therefore postponed to Section \ref{sec.asym_behav_Y_n} to avoid interrupting the main exposition. Note that $\< \hat{V}^{n,k_1, \ell}, \hat{V}^{n,k_2, q} \> = 0$ and $\< \hat{V}^{n,k, \ell}, W^q \> = 0$ when $\ell \neq q$. Then, Lemma \ref{lm.covV} and the theory of Jacod \cite{Jacod1997} allow us to specify the limiting distribution of $\widehat{\mathbf{V}}^n$ in the following lemma.

\begin{lemma} \label{lem.vklB}
As $n$ goes to infinity, the process $\widehat{\mathbf{V}}^{n}$ converges stably in law in $\mathcal{C}_{0}$ to a continuous process $\widehat{\mathbf{V}} = \{\hat{V}^{k, \ell}\}$ with 
\begin{align*}
\hat{V}^{k, \ell} = \kappa_{1}(\alpha) \sum_{q=1}^m \int_0^\cdot \sigma_q^k(X_s) \rd B_s^{q,\ell}, \qquad 1 \leq k \leq d, \quad 1 \leq \ell \leq m. 
\end{align*}
Here, $B$ is an $m^2$-dimensional standard Brownian motion independent of \(\cF\) and defined on some extension of $(\Omega, \cF, \hP)$, and $\kappa_{1}(\alpha)$ is defined in \eqref{eq.kappa1}. 
\end{lemma}

With Proposition \ref{pro.X} and Appendix \ref{appen:ML function} at hand, the following lemma can be established as an extension of \cite[Lemmas 2.4--2.8]{FukasawaUgai2023}. For conciseness, its proof is omitted.

\begin{lemma} \label{lm.group}
The following statements hold. 
\begin{itemize}
\item[(i)] For any $(i, j) \in \{1, \dots, d\}^2$ and $\epsilon \in (0, \alpha - \frac{1}{2})$, 
\begin{align*}
\int_0^t K^{i}_{j}(t-s) n^{\alpha-\frac{1}{2}} \nabla b^j(\widetilde{X}_{s})^\top (\widetilde{X}_s - \widehat{X}_{\underline{s}}) \rd s 
\ \xrightarrow[n \to \infty]{\hP}\ 
0 \qquad \mbox{in}~~\mathcal{C}^{\alpha-\frac{1}{2}-\epsilon}_0.
\end{align*}
\vskip 0.3em

\item[(ii)] As $n$ goes to infinity, $\|\widehat{\Delta}^n\|_{C_0^\gamma}$ tends to zero in $L^p(\Omega)$ for all $\gamma \in (0, \alpha-\frac{1}{2})$ and $p\geq1$.
\vskip 0.3em

\item[(iii)] If the sequence 
\begin{align*}
 \big( Y^n, \, \widehat{\mathbf{V}}^{n}, \, \{\nabla b^j(\widetilde{X})\}_j, \, \{\partial_k \sigma^j_{\ell}(\widetilde{X})\}_{j \ell k} \big) 
\end{align*}
converges in law in $\mathcal{C}_0^{\alpha - \frac{1}{2} - \epsilon} \times \mathcal{C}_0 \times \mathcal{D}_{d^2} \times \mathcal{D}_{d^2 m}$ to
\begin{align*}
 \big( Y, \, \widehat{\mathbf{V}}, \, \{\nabla b^j(X)\}_j, \, \{\partial_k \sigma^j_{\ell}(X)\}_{j \ell k} \big) \qquad \text{as}~~ n \rightarrow \infty, 
\end{align*}
then $Y$ is a solution of \eqref{eq.Y_t^i}.
\vskip 0.3em

\item[(iv)] The sequence $Y^n$ is tight in $\mathcal{C}_0^{\alpha - \frac{1}{2} -\epsilon}$ for all $\epsilon \in (0, \alpha - \frac{1}{2})$.
\vskip 0.3em

\item[(v)] The uniqueness in law holds for the solution of Eq.\ \eqref{eq.Y_t^i}.
\end{itemize}
\end{lemma}

With Lemmas \ref{lm.covV}--\ref{lm.group} established and drawing on the approach of \cite{FukasawaUgai2023}, we now proceed with the proof of Theorem \ref{thm.Yn}.

\vskip 0.5em
\textit{Proof of Theorem \ref{thm.Yn}}. 
In view of \eqref{X-tildeX_LpC}, $\widetilde{X} \xrightarrow[n \to \infty]{\hP} X$ in the uniform topology. Then it follows from the continuous mapping theorem that 
\begin{align*}
\big( \{\nabla b^j(\widetilde{X})\}_j,\, \{\partial_k \sigma_{\ell}^j(\widetilde{X})\}_{j\ell k} \big) 
\ \xrightarrow[n \to \infty]{\hP}\ 
\big( \{\nabla b^j(X)\}_j,\, \{\partial_k \sigma_{\ell}^j(X)\}_{j\ell k} \big)
\end{align*}
in the uniform topology as well. Recalling Lemmas \ref{lem.vklB} and \ref{lm.group}(iv) concludes that
\begin{align*}
\big(Y^n,\, \widehat{\mathbf{V}}^{n},\, \{\nabla b^j(\widetilde{X})\}_j,\, \{\partial_k \sigma_{\ell}^j(\widetilde{X})\}_{j\ell k},\, H \big)
\end{align*} 
is tight in $C_0^{\alpha-\frac12-\epsilon} \times C_0 \times D_{d^2} \times D_{d^2 m} \times \hR$ for any random variable $H$ on $(\Omega, \cF, \hP)$. By Prokhorov's theorem (see e.g., \cite[Theorem 13.29]{Klenke2020}), any subsequence of this tight sequence admits a further subsequence that converges. Moreover, Lemma \ref{lm.group}(iii) and (v) show the uniqueness of the limit. Therefore the original sequence itself has to converge. Based on Lemma \ref{lm.group}(v), the limit $Y$ of $Y^n$ is characterized by \eqref{eq.Y_t^i}. Finally, the convergence of $Y^n$ is stable because $H$ is arbitrary. The proof is completed. 
\hfill$\Box$

\subsection{Proof of Theorem \ref{thm.An}}

For the second part $R^{n}$ of $U^{n}$, it follows from \eqref{eq.MLE} and \eqref{eq.auxiMLE} that for any $t \in [0,T]$, 
\begin{align*} 
R_t^n = n^{\alpha - \frac{1}{2}} \int_0^t \big( K(t-s) - K(t-\underline{s}) \big) b(\widehat{X}_{\underline{s}}) \rd s + n^{\alpha - \frac{1}{2}} \int_0^t \big( K(t-s) - K(t-\underline{s}) \big) \sigma(\widehat{X}_{\underline{s}}) \rd W_{s}. 
\end{align*} 
For $t \in [0, T ]$, define 
\begin{align}
& C_t^{n} 
= n^{\alpha - \frac{1}{2}} \int_0^t \big( K(t-s) - K(t-\underline{s}) \big) b(\widehat{X}_{\underline{s}}) \rd s, \label{eq.def_C_tn} \\ 
& \widehat{R}_{t}^{n} 
= n^{\alpha - \frac{1}{2}} \int_0^t \big( K(t-s) - K(t-\underline{s}) \big) \sigma(\widehat{X}_{\underline{s}}) \rd W_s, \label{eq.def_R_tn1} \\ 
& \widetilde{R}_{t}^{n} 
= n^{\alpha - \frac{1}{2}} \int_0^t \big( K(t-s) - K(t-\underline{s}) \big) \sigma(X_t) \rd W_s \label{eq.def_tildR_tn1}.
\end{align}

To facilitate the proof of Theorem \ref{thm.An}, we prepare Lemmas \ref{lem.replaceR_n1} and \ref{lm.R_n1}, whose proofs are placed in Section \ref{sec.asym_behav_R_n}. The following Lemma \ref{lem.replaceR_n1}(ii) shows that $\widehat{R}_{t}^{n}$ exhibits the same asymptotic behavior as $\widetilde{R}_{t}^{n}$ in $L^2(\Omega; \hR^d)$ for all $t \in [0,T]$.

\begin{lemma} \label{lem.replaceR_n1}
We have the following two statements:\ 
\begin{itemize}
\item[(i)] $C_t^{n} \xrightarrow[n \to \infty]{L^2(\Omega;\, \hR^d)} 0$ for any $t \in [0, T]$. 
\item[(ii)] $\lim_{n \to \infty} \sup_{t \in [0, T]} \hE \big[ \big |\widehat{R}_{t}^{n} - \widetilde{R}_{t}^{n}\big|^2 \big] = 0$. 
\end{itemize}
\end{lemma}

\begin{lemma} \label{lm.R_n1}
As $n$ goes to infinity, the finite-dimensional distributions of the process $\{\widetilde{R}^{n}_{\underline{t}},t\in(0,T]\}$ converge in distribution to those of $\{ \widetilde{\cR}_t,t\in(0,T]\}$, whose definition can be found in \eqref{eq.def:widetilde_cR}. 
\end{lemma}

\vskip 0.5em 
\textit{Proof of Theorem \ref{thm.An}}.
On one hand, it follows from Lemma \ref{lem.replaceR_n1}(i) that for any $t_{1}, \dots, t_{m} \in (0, T]$, 
\begin{align*}
(C_{t_1}^n, \dots, C_{t_m}^n) \xrightarrow[n \to \infty]{\hP} 
0. 
\end{align*} 
On the other hand, it follows from Lemma \ref{lem.replaceR_n1}(ii) that 
\begin{align*}
\hE \left[\left| (\widehat{R}_{t_1}^{n}, \dots, \widehat{R}_{t_m}^{n}) - (\widetilde{R}_{t_1}^{n}, \dots, \widetilde{R}_{t_m}^{n}) \right|^2 \right]\xrightarrow[n \to \infty]{} 
0,
\end{align*}
which together with Lemma \ref{lm.R_n1} and Slutsky's theorem (see e.g., \cite[Theorem 13.18]{Klenke2020}) indicates that for any $t_{1}, \dots, t_{m} \in (0, T]$, 
\begin{align*}
(\widehat{R}_{\underline{t_1}}^{n}, \dots, \widehat{R}_{\underline{t_m}}^{n}) 
\xrightarrow[n \to \infty]{d} 
(\widetilde{\cR}_{t_1}, \cdots, \widetilde{\cR}_{t_m}).
\end{align*} 
Thus, by Slutsky's theorem again, one can conclude that 
\begin{align*}
(R_{\underline{t_1}}^n, \dots, R_{\underline{t_m}}^n) = (C_{\underline{t_1}}^n + \widehat{R}_{\underline{t_1}}^{n}, \dots, C_{\underline{t_m}}^n + \widehat{R}_{\underline{t_m}}^{n}) \xrightarrow[n \to \infty]{d} 
(\widetilde{\cR}_{t_1}, \cdots, \widetilde{\cR}_{t_m}), 
\end{align*}
which completes the proof. 
\hfill$\Box$

\section{Proof of Lemma \ref{lm.covV}} 
 \label{sec.asym_behav_Y_n}

We establish the following auxiliary lemmas, which serve as crucial ingredients in the proof of Lemma \ref{lm.covV}.

\begin{lemma} \label{lm.cB_n}
Let $(i, j) \in \{1, 2, \cdots, d\}^2$. For $0\leq v < s\leq T$, define 
\begin{align*}
\cB_n(v, s) 
= n^{2\alpha-1} \int_0^{\underline{v}} \left| \big( K^{i}_{j} (s-u) - K^{i}_{j}(\underline{s} - \underline{u}) \big) \big( K^{i}_{j} (v-u) - K^{i}_{j}(\underline{v} - \underline{u}) \big) \right| \rd u. 
\end{align*} 
Then, $\sup_{0 \leq v < s \leq T} \sup_n \cB_n(v, s) < \infty$ and $\lim_{n \to \infty} \cB_n(v, s) = 0$. 
\end{lemma}

\begin{proof} 
We define the dominating function $\Psi: (0, \infty) \rightarrow (0, \infty)$ by
\begin{align} \label{eq.defPsi} 
\Psi(y) = 
\begin{cases} 
y^{2\alpha-2}, & y \in (0, 1], \\
y^{2\alpha-4}, & y \in (1, \infty), 
\end{cases}
\end{align}
which satisfies that $\int_{0}^{\infty} \Psi(y) \rd y \leq C$ since $\alpha \in (\frac{1}{2},1)$. By $K(\cdot) = \cK(\cdot) \cE(\cdot)$, \eqref{eq.def:k&e} and the change of variables $y = \lfloor nv \rfloor - nu$, one has 
\begin{align}
\cB_n(v, s) 
&= n \int_0^{\underline{v}} \left| \cK(ns-nu) \cE^{i}_j(s-u) - \cK( \lfloor ns \rfloor - \lfloor nu \rfloor) \cE^{i}_j(\underline{s} - \underline{u}) \right| \notag \\ 
&\qquad\quad \times \left| \cK(nv-nu) \cE^{i}_j(v-u) - \cK( \lfloor nv \rfloor - \lfloor nu \rfloor) \cE^{i}_j\left(\underline{v} - \underline{u} \right) \right| \rd u 
= \int_0^{ \lfloor nv \rfloor} g^n(y) \rd y, \label{eq.Bn_vs_bound}
\end{align}
where 
\begin{align*}
g^n(y) 
&:= \biggr| \cK(y + ns- \lfloor nv \rfloor) \cE^{i}_j\left(\frac{y + ns- \lfloor nv \rfloor}{n}\right)- \cK(\lceil y\rceil + \lfloor ns \rfloor - \lfloor nv \rfloor) \cE^{i}_j\left(\frac{\lceil y\rceil + \lfloor ns \rfloor - \lfloor nv \rfloor}{n}\right)
 \biggr | \\ 
&\quad\ \times \biggr| \cK(y + nv- \lfloor nv \rfloor) \cE^{i}_j \left(\frac{y + nv- \lfloor nv \rfloor}{n} \right) 
- \cK(\lceil y\rceil) \cE^{i}_j\left(\frac{\lceil y\rceil}{n} \right) \biggr|. 
\end{align*} 
When $y \in (0,1]$, it follows from \eqref{eq.cU_Holder} that 
\begin{align*}
g^n(y) 
&\leq C \big(\cK(y + ns- \lfloor nv \rfloor) + \cK(\lceil y\rceil + \lfloor ns \rfloor- \lfloor nv \rfloor) \big) \big(\cK(y + nv- \lfloor nv \rfloor) + \cK(\lceil y \rceil) \big) \\ 
&\leq C \cK^2(y) = C \Psi(y). 
\end{align*}
When $y \in (1,\infty)$, it follows from \eqref{eq.MLder} and \eqref{eq.cU_Holder} that 
\begin{align*}
g^n(y) 
&\leq \left| \int_{\lceil y\rceil + \lfloor ns \rfloor - \lfloor nv \rfloor}^{y + ns- \lfloor nv \rfloor} \rd \left(x^{\alpha-1} E^{i,j}_{\alpha, \alpha} \left(\frac{A}{n^{\alpha}}x^{\alpha} \right )\right) \right| \left| \int_{\lceil y\rceil}^{y + nv- \lfloor nv \rfloor} \rd \left(x^{\alpha-1} E^{i,j}_{\alpha, \alpha} \left(\frac{A}{n^{\alpha}}x^{\alpha} \right)\right) \right| \\
&\leq \int_{\lceil y\rceil + \lfloor ns \rfloor - \lfloor nv \rfloor}^{y + ns- \lfloor nv \rfloor} y^{\alpha-2} \left| E^{i,j}_{\alpha, \alpha-1} \left(\frac{A}{n^{\alpha}}x^{\alpha} \right) \right| \rd x\, \int_{\lceil y\rceil}^{y + nv- \lfloor nv \rfloor} y^{\alpha-2} \left| E^{i,j}_{\alpha, \alpha-1} \left(\frac{A}{n^{\alpha}}x^{\alpha} \right) \right| \rd x \\ 
&\leq C y^{2\alpha-4} = C \Psi(y). 
\end{align*}
Thus, applying \eqref{eq.Bn_vs_bound}, one obtains 
\begin{align} \label{eq.g_nBound}
\sup_{0 \leq v < s \leq T} \sup_n \cB_n(v, s)\leq\int_0^{\infty} g^n(y) \rd y 
\leq C \int_0^{\infty} \Psi(y) \rd y \leq C,
\end{align}
where $C>0$ independent of $n$, $v$ and $s$.

For any $u\in(0,\infty)$, by the fact that $\cE\left(\frac {u}{n}\right) \xrightarrow[n\to\infty]{} \frac{1}{\Gamma(\alpha)} I_{d \times d}$ (where $I_{d \times d}$ denotes the $d \times d$ identity matrix), one gets
\begin{align} \label{eq.limeij}
\lim_{n\to \infty} \cE^{i}_{j}\left( \frac{u}{n} \right) = \frac{\delta_{i,j}}{\Gamma(\alpha)} = \begin{cases} \frac{1}{\Gamma(\alpha)}, & i = j, \\ 0, & i \neq j, \end{cases}
\end{align}
where $\delta_{i,j}$ is the Kronecker delta. For any $0\leq v < s\leq T$ and $y \in (0, \lfloor nv \rfloor)$, using \eqref{eq.cU_Holder} and the monotonicity of $\cK(\cdot)$ shows 
\begin{align*}
&\ \bigg| \cK(y + ns- \lfloor nv \rfloor) \cE^{i}_j \left(\frac{y + ns- \lfloor nv \rfloor}{n} \right)- \cK( \lceil y\rceil + \lfloor ns \rfloor - \lfloor nv \rfloor) \cE^{i}_j\left(\frac{\lceil y\rceil + \lfloor ns \rfloor - \lfloor nv \rfloor}{n} \right) \bigg| \\
&\leq C \big( \cK(y + ns- \lfloor nv \rfloor) + \cK( \lceil y\rceil + \lfloor ns \rfloor - \lfloor nv \rfloor) \big) 
\xrightarrow[n \to \infty]{} 0 
\end{align*} 
and 
\begin{align*}
\left|\cK(y + nv- \lfloor nv \rfloor) \cE^{i}_j \left(\frac{y + nv- \lfloor nv \rfloor}{n} \right) - \cK( \lceil y\rceil) \cE^{i}_j\left(\frac{\lceil y\rceil}{n} \right)\right|\leq C\cK(y).
\end{align*} 
Further by the dominated convergence theorem (DCT), $\int_0^\infty g^n(y) \rd y \xrightarrow[n \to \infty]{} 0$, which together with \eqref{eq.Bn_vs_bound}, implies that $\lim_{n \to \infty} \cB_n(v, s) = 0$. The proof of Lemma \ref{lm.cB_n} is completed. 
\end{proof}

\begin{lemma} \label{lm.cJ_n}
Let $(i_1,i_2,j_1,j_2) \in \{1,2,\cdots, d\}^4$. For $s \in [0,T]$, $y \in (0,\infty)$ and integer $n \geq 1$, define 
\begin{align} \label{eq.F^n(s,y)}
(\cJ_n)^{i_1,i_2}_{j_1,j_2}(s,y) 
&:= \left( \cK(y + ns - \lfloor ns \rfloor) \cE^{i_1}_{j_1}\left(\frac{y + ns- \lfloor ns \rfloor}{n} \right) - \cK(\lceil y \rceil) \cE^{i_1}_{j_1} \left(\frac{\lceil y\rceil}{n} \right) \right) \nonumber \\ 
&\qquad \times \left(\cK(y + ns - \lfloor ns \rfloor) \cE^{i_2}_{j_2}\left(\frac{y + ns- \lfloor ns \rfloor}{n} \right) - \cK(\lceil y \rceil) \cE^{i_2}_{j_2} \left(\frac{\lceil y\rceil}{n} \right) \right) 
\end{align}
and
\begin{align*}
(\cG_n)^{i_1,i_2}_{j_1,j_2}(s,y) 
:=
\begin{cases}
\frac{1}{\Gamma^{2}(\alpha)} \big(\cK(y + ns - \lfloor ns \rfloor) - \cK(\lceil y \rceil) \big)^{2},& i_1 = j_1 \text{and } i_2 = j_2, \\ 0, & \text{otherwise}. 
 \end{cases} 
 \end{align*}
Then for any $t \in [0,T]$, it holds that 
\begin{align} 
& \int_0^t \int_0^{\infty} 1_{( \lfloor ns \rfloor, \infty)} (y) (\cJ_n)^{i_1,i_2}_{j_1,j_2}(s,y) \sigma_{\ell}^{j_1}(X_{s}) \sigma_{\ell}^{j_2}(X_{s}) \rd y \rd s \xrightarrow[n \to \infty]{L^2(\Omega)} 0 \label{eq.ns_to_infty_for_cJ_n}, \\
& \int_0^t \int_0^{\infty} \big( (\cJ_n)^{i_1,i_2}_{j_1,j_2}(s,y) - (\cG_n)^{i_1,i_2}_{j_1,j_2}(s,y) \big) \sigma_{\ell}^{j_1}(X_{s}) \sigma_{\ell}^{j_2}(X_{s}) \rd y \rd s \xrightarrow[n \to \infty]{L^2(\Omega)} 0. \label{eq.E to 1}
\end{align} 
\end{lemma}

\begin{proof}
Using Minkowski's integral inequality yields
\begin{align*} 
&\ \bigg\| \int_0^t \int_0^{\infty} 1_{( \lfloor ns \rfloor, \infty)} \left(y \right) (\cJ_n)^{i_1,i_2}_{j_1,j_2}(s,y) \sigma_{\ell}^{j_1}(X_{s}) \sigma_{\ell}^{j_2}(X_{s}) \rd y \rd s \bigg\|_{L^{2}} \nonumber \\ 
&\leq \int_0^t \int_0^{\infty} 1_{( \lfloor ns \rfloor, \infty)} \left(y \right) \big| (\cJ_n)^{i_1,i_2}_{j_1,j_2}(s,y) \big| \big\| \sigma_{\ell}^{j_1}(X_{s}) \sigma_{\ell}^{j_2}(X_{s}) \big\|_{L^{2}} \rd y \rd s. 
\end{align*}
By the linear growth of $\sigma$ and \eqref{eq.momentBounded}, one gets 
\begin{align} \label{eq.sigma^2}
\big\| \sigma_{\ell}^{j_1}(X_{s}) \sigma_{\ell}^{j_2}(X_{s}) \big\|_{L^{2}} 
\leq C.
\end{align}
By a similar argument to that used for \eqref{eq.g_nBound}, one also has 
\begin{align} \label{eq.cJn<Psi}
\big| (\cJ_n)^{i_1,i_2}_{j_1,j_2} (s,y) \big| + \big| (\cG_n)^{i_1,i_2}_{j_1,j_2} (s,y) \big| 
\leq C \Psi (y), \qquad \forall\, s \in [0,T], \quad y \in (0,\infty), 
\end{align} 
where $\Psi$ is defined in \eqref{eq.defPsi}, and $C > 0$ is independent of $n$. Then, by applying the DCT w.r.t.\ $ \rd y\otimes \rd s$, one can arrive at \eqref{eq.ns_to_infty_for_cJ_n}.

Using Minkowski's integral inequality and \eqref{eq.sigma^2} shows
\begin{align*} 
&\ \bigg\| \int_0^t \int_0^{\infty} \big( (\cJ_n)^{i_1,i_2}_{j_1,j_2}(s,y) - (\cG_n)^{i_1,i_2}_{j_1,j_2}(s,y) \big) \sigma_{\ell}^{j_1}(X_{s}) \sigma_{\ell}^{j_2}(X_{s}) \rd y \rd s \bigg\|_{L^{2}} \nonumber \\ 
&\leq C \int_0^t \int_0^{\infty} \big| (\cJ_n)^{i_1,i_2}_{j_1,j_2}(s,y) - (\cG_n)^{i_1,i_2}_{j_1,j_2}(s,y) \big| \rd y \rd s. 
\end{align*} 
Note that $\big| (\cJ_n)^{i_1,i_2}_{j_1,j_2}(s,y) - (\cG_n)^{i_1,i_2}_{j_1,j_2}(s,y) \big| \xrightarrow[n \to \infty]{} 0$ by \eqref{eq.limeij}. Thus, by recalling \eqref{eq.cJn<Psi} and applying the DCT w.r.t.\ $ \rd y\otimes \rd s$, one can obtain \eqref{eq.E to 1}. The proof of Lemma \ref{lm.cJ_n} is completed. 
\end{proof}

\subsection{Proof of \eqref{lm.covV(i)}}

Using \eqref{eq.MLE} and \eqref{eq.auxiMLE} yields that for $i \in \{1,\ldots,d\}$ and $s \in [0,T]$, 
\begin{align*} 
\widetilde{X}_s^i - \widehat{X}_{\underline{s}}^i 
= \sum_{j=1}^d \left( \psi^{i,j}_{n,0}(s) + \psi^{i,j}_{n, 1}(s) + \psi^{i,j}_{n, 2}(s) \right)
\end{align*}
with $\psi^{i,j}_{n,0}(s) := (E_{\alpha}^{i,j} (A s^{\alpha}) -E_{\alpha}^{i,j}(A\underline{s}^{\alpha})) X^{j}_{0}$, 
\begin{align*} 
\psi^{i,j}_{n, 1}(s) 
&:= \int_0^{s} K^{i}_{j}(s-u) b^j(\widehat{X}_{\underline{u}}) \rd u - \int_{0}^{\underline{s}} K^{i}_{j}({\underline{s}}-{\underline{u}}) b^j(\widehat{X}_{\underline{u}}) \rd u \\ 
&\, = \int_0^{\underline{s}} \big( K^{i}_{j}(s-u)- K^{i}_{j}({\underline{s}}-{\underline{u}}) \big) b^j(\widehat{X}_{\underline{u}}) \rd u 
+ \int_{\underline{s}}^s K^{i}_{j} (s-u) b^j(\widehat{X}_{\underline{u}}) \rd u, \\
\psi^{i,j}_{n, 2}(s) 
&:= \sum_{\ell=1}^{m} \int_{0}^{s} K^{i}_{j}(s-u) \sigma_{\ell}^{j}(\widehat{X}_{\underline{u}}) \rd W_{u}^{\ell} - \sum_{\ell=1}^{m} \int_{0}^{\underline{s}} K^{i}_{j}(\underline{s}-\underline{u}) \sigma_{\ell}^{j}(\widehat{X}_{\underline{u}}) \rd W_{u}^{\ell} \\ 
&\, = \sum_{\ell=1}^{m} \int_{0}^{\underline{s}} \big( K^{i}_{j}(s-u) - K^{i}_{j}(\underline{s} - \underline{u}) \big) \sigma_{\ell}^{j}(\widehat{X}_{\underline{u}}) \rd W_{u}^{\ell} + \sum_{\ell=1}^{m} \int_{\underline{s}}^{s} K^{i}_{j}(s-u) \sigma_{\ell}^{j}(\widehat{X}_{\underline{u}}) \rd W_{u}^{\ell}. 
\end{align*} 
Then, one gets the formula 
\begin{align*}
&\ \< \hat{V}^{n,k_1, \ell}, \hat{V}^{n,k_2, \ell} \>_t 
= n^{2\alpha-1} \int_0^t ( 
\widetilde{X}^{k_1}_s - \widehat{X}^{k_1}_{\underline{s}})(\widetilde{X}^{k_2}_s - \widehat{X}^{k_2}_{\underline{s}}) \rd s \\ 
&= n^{2\alpha-1} \int_0^t \bigg( \sum_{j_1=1}^d \psi^{k_1,j_1}_{n,0}(s) (\widetilde{X}^{k_2}_s - \widehat{X}^{k_2}_{\underline{s}}) + \sum_{j_1=1}^d \psi^{k_1,j_1}_{n, 1}(s) (\widetilde{X}^{k_2}_s - \widehat{X}^{k_2}_{\underline{s}})\\
&\quad + \sum_{j_1=1}^d\sum_{j_2=1}^d \psi^{k_1,j_1}_{n, 2}(s) \psi^{k_2,j_2}_{n,0}(s) + \sum_{j_1=1}^d \sum_{j_2=1}^d \psi^{k_1,j_1}_{n, 2}(s) \psi^{k_2,j_2}_{n, 1}(s) + \sum_{j_1=1}^d \sum_{j_2=1}^d 
\psi^{k_1,j_1}_{n, 2}(s) \psi^{k_2,j_2}_{n, 2}(s) \bigg) \rd s. 
\end{align*}
Thus, the proof of \eqref{lm.covV(i)} can be completed by \eqref{X-hatX_Lp} and Lemma \ref{lm.main} below. 
\hfill$\Box$

\begin{lemma} \label{lm.main}
For any $(i, j) \in \{1,\ldots,d\}^2$, it holds that 
\begin{align}
\lim_{n \to \infty} \sup_{s \in [0, T]} \left\|n^{\alpha-\frac{1}{2}} \psi^{i,j}_{n,0}(s) \right\|_{L^2} &= 0, \label{eq.EX0}\\
\lim_{n \to \infty} \sup_{s \in [0, T]} \big\| n^{\alpha-\frac{1}{2}} \psi^{i,j}_{n, 1}(s) \big\|_{L^2} &= 0, \label{eq.psi_{1, s}} \\
\sup_{n \geq 1} \sup_{s \in [0, T]} \big\| n^{\alpha-\frac{1}{2}} \psi^{i,j}_{n, 2}(s) \big\|_{L^2} &< \infty. \label{eq.psi_{2, s}}
\end{align} 
Moreover, for any $(i_1,j_1,i_2,j_2) \in \{1,\ldots,d\}^4$ and $t \in [0, T]$, 
\begin{align} \label{eq.psi_{2, s}^2}
n^{2\alpha-1} \sum_{j_1=1}^d \sum_{j_2=1}^d \int_0^t \psi^{i_1,j_1}_{n,2}(s) \psi^{i_2,j_2}_{n,2}(s) \rd s 
\ \xrightarrow[n \to \infty]{L^2(\Omega)}\ 
\kappa_{1}^{2}(\alpha) \sum_{\ell=1}^m \int_0^t \sigma_{\ell}^{i_1}(X_s) \sigma_{\ell}^{i_2}(X_s) \rd s, 
\end{align}
where the constant $\kappa_{1}^{2}(\alpha)$ is defined in \eqref{eq.kappa1}. 
\end{lemma}

\begin{proof}
Based on \eqref{eq.momentBounded} and the fact that $\sup_{s \in [0, T]} |E_{\alpha} (A s^{\alpha}) -E_{\alpha} (A\underline{s}^{\alpha})| \leq C n^{-\alpha}$, one has \eqref{eq.EX0}. Using \eqref{lm.|K-K|H}, \eqref{lm.|K-K|Z}, \eqref{eq.momentBounded} and the linear growth conditions of $b$ and $\sigma$ deduces 
\begin{align*}
\| \psi^{i,j}_{n, 1}(s) \|_{L^2} \leq C n^{-\alpha} \qquad \mbox{and} \qquad 
\| \psi^{i,j}_{n, 2}(s) \|_{L^2} \leq C n^{\frac{1}{2}-\alpha}, 
\end{align*}
which imply \eqref{eq.psi_{1, s}} and \eqref{eq.psi_{2, s}}, respectively.

Next we prove \eqref{eq.psi_{2, s}^2}. To facilitate the analysis, we rewrite $\psi^{i,j}_{n, 2}(s)$ as
\begin{align*}
\psi^{i,j}_{n, 2}(s)
= \sum_{\ell=1}^{m} \int_{0}^{\underline{s}} f_{\ell}^{i,j}(s, u) \rd W_{u}^{\ell} + \sum_{\ell=1}^{m} \int_{\underline{s}}^{s} K^{i}_{j}(s-u) \sigma_{\ell}^{j}(\widehat{X}_{\underline{s}}) \rd W_{u}^{\ell}, 
\end{align*}
where $ f_{\ell}^{i,j}(s, u) := \big( K^{i}_{j}(s-u) - K^{i}_{j}(\underline{s} - \underline{u}) \big) \sigma_{\ell}^{j}(\widehat{X}_{\underline{u}})$. It helps us to make following the decomposition
\begin{align*}
n^{2\alpha-1} \int_0^t \psi^{i_1,j_1}_{n,2}(s) \psi^{i_2,j_2}_{n,2}(s) \rd s 
= \sum_{\ell,q=1}^{m} \left((\mathrm{I})_{\ell,q}^{i_1,j_1,i_2,j_2} + (\mathrm{II})_{\ell,q}^{i_1,j_1,i_2,j_2} + (\mathrm{III})_{\ell,q}^{i_1,j_1,i_2,j_2} + (\mathrm{IV})_{\ell,q}^{i_1,j_1,i_2,j_2}\right), 
\end{align*} 
where 
\begin{align*}
& (\mathrm{I})_{\ell,q}^{i_1,j_1,i_2,j_2} := n^{2\alpha-1} \int_{0}^{t} \left( \int_{0}^{\underline{s}} f_{\ell}^{i_1,j_1}(s, u) \rd W_{u}^{\ell} \right) \left( \int_{0}^{\underline{s}} f_{q}^{i_2,j_2}(s, u) \rd W_{u}^{q} \right) \rd s, \\ 
& (\mathrm{II})_{\ell,q}^{i_1,j_1,i_2,j_2} := n^{2\alpha-1} \int_{0}^{t} \sigma_{q}^{j_{2}}(\widehat{X}_{\underline{s}}) \left( \int_{0}^{\underline{s}} f_{\ell}^{i_1,j_1}(s, u) \rd W_{u}^{\ell} \right) \left( \int_{\underline{s}}^{s} K^{i_2}_{j_2}(s-u) \rd W_{u}^{q} \right) \rd s, \\ 
& (\mathrm{III})_{\ell,q}^{i_1,j_1,i_2,j_2} := n^{2\alpha-1} \int_0^t \sigma_{\ell}^{j_1}(\widehat{X}_{\underline{s}}) \left( \int_0^{\underline{s}} f_{q}^{i_2,j_2}(s, u) \rd W_u^q \right) \left( \int_{\underline{s}}^s K^{i_1}_{j_1}(s-u) \rd W_u^{\ell} \right) \rd s, \\ 
& (\mathrm{IV})_{\ell,q}^{i_1,j_1,i_2,j_2} := n^{2\alpha-1} \int_0^t \sigma_{\ell}^{j_1}(\widehat{X}_{\underline{s}}) \sigma_q^{j_2}(\widehat{X}_{\underline{s}}) \left( \int_{\underline{s}}^s K^{i_1}_{j_1}(s-u) \rd W_u^{\ell} \right) \left( \int_{\underline{s}}^s K^{i_2}_{j_2}(s-u) \rd W_u^q \right) \rd s.
\end{align*} 
One can claim that 
\begin{align}
&(\mathrm{I})_{\ell,q}^{i_1,j_1,i_2,j_2} 
\xrightarrow[n \to \infty]{L^2(\Omega)} 
\begin{cases} 
\kappa_{1,1}^{2}(\alpha) \int_{0}^{t} \sigma_{\ell}^{i_1}(X_{s}) \sigma_{\ell}^{i_2}(X_{s}) \rd s, & \ell = q,\,i_1=j_1,\text{ and }i_2=j_2, \\ 
0, &\text {otherwise}, 
\end{cases} \label{eq.(17)} \\
&(\mathrm{II})_{\ell,q}^{i_1,j_1,i_2,j_2} 
\xrightarrow[n \to \infty]{L^2(\Omega)} 
0, \label{eq.(18)} \\
&(\mathrm{III})_{\ell,q}^{i_1,j_1,i_2,j_2} 
\xrightarrow[n \to \infty]{L^2(\Omega)} 
0, \label{eq.(19)} \\
&(\mathrm{IV})_{\ell,q}^{i_1,j_1,i_2,j_2} 
\xrightarrow[n \to \infty]{L^2(\Omega)} 
\begin{cases}
\frac{1}{\Gamma^{2}(\alpha)2\alpha(2\alpha-1)} \int_{0}^{t} \sigma_{\ell}^{i_1}(X_{s}) \sigma_{\ell}^{i_2}(X_{s}) \rd s, & \ell = q,\,i_1=j_1,\text{ and }i_2=j_2, \\
0, &\text {otherwise}. 
\end{cases} \label{eq.(20)}
\end{align}
Here, the constant $\kappa_{1,1}^{2}(\alpha)$ is defined by
\begin{align} \label{eq.kappa_1,1^2}
\kappa_{1,1}^{2}(\alpha) 
= \frac{1}{\Gamma^{2}(\alpha)}\int_{0}^{1} \int_{0}^{\infty} ((y + x)^{\alpha-1}- \left( \lceil y \rceil \right)^{\alpha-1})^{2} \rd y \rd x. 
\end{align}
Further, it follows from \eqref{eq.(17)}--\eqref{eq.(20)} that \eqref{eq.psi_{2, s}^2} holds. Thus, it remains to show that \eqref{eq.(17)}--\eqref{eq.(20)} hold.

\underline{\emph{Proof of \eqref{eq.(17)}}}. 
According to It\^o’s product rule, for any progressively measurable and square integrable processes $h_{1}$ and $h_{2}$, it holds that 
\begin{align} 
&\ \left( \int_{s}^{t}h_{1}(u) \rd W_{u}^{\ell} \right) \left( \int_{s}^{t}h_{2}(u) \rd W_{u}^{q} \right) 
= \int_{s}^{t} \left( \int_{s}^{u}h_{1}(r) \rd W_{r}^{\ell} \right)h_{2}(u) \rd W_{u}^{q} \nonumber\\ 
&\qquad\qquad + \int_{s}^{t} \left( \int_{s}^{u}h_{2}(r) \rd W_{r}^{q} \right)h_{1}(u) \rd W_{u}^{\ell} + \int_{s}^{t}h_{1}(u)h_{2}(u) \rd \< W^{\ell}, W^q \>_{u}. \label{eq.ItoProduct Rule}
\end{align}
Then, $(\mathrm{I})_{\ell,q}^{i_1,j_1,i_2,j_2}$ can be split into 
\begin{align*}
(\mathrm{I})_{\ell,q}^{i_1,j_1,i_2,j_2}
= (\mathrm{I})_{\ell,q, 1}^{i_1,j_1,i_2,j_2} + (\mathrm{I})_{\ell,q, 2}^{i_1,j_1,i_2,j_2} + (\mathrm{I})_{\ell,q, 3}^{i_1,j_1,i_2,j_2}
\end{align*}
with 
\begin{align*}
&(\mathrm{I})_{\ell,q, 1}^{i_1,j_1,i_2,j_2} := n^{2\alpha-1} \int_0^t \int_0^{\underline{s}} \left( \int_0^u f_{\ell}^{i_1,j_1}(s, r) \rd W_r^{\ell} \right) f_{q}^{i_2,j_2}(s, u) \rd W_u^q \rd s, \\
&(\mathrm{I})_{\ell,q, 2}^{i_1,j_1,i_2,j_2} := n^{2\alpha-1} \int_0^t \int_0^{\underline{s}} \left( \int_0^u f_{q}^{i_2,j_2}(s, r) \rd W_r^q \right) f_{\ell}^{i_1,j_1}(s, u) \rd W_u^{\ell} \rd s, \\ 
&(\mathrm{I})_{\ell,q, 3}^{i_1,j_1,i_2,j_2} := n^{2\alpha-1} \int_0^t \int_0^{\underline{s}} f_{\ell}^{i_1,j_1}(s, u) f_{q}^{i_2,j_2}(s, u) \rd \< W^{\ell}, W^q \>_u \rd s. 
\end{align*} 
For $(\mathrm{I})_{\ell,q, 1}^{i_1,j_1,i_2,j_2}$, setting 
\begin{align*}
D_{1, s}^{\ell,q} := n^{2\alpha-1} \int_0^{\underline{s}} \left( \int_0^u f_{\ell}^{i_1,j_1}(s, r) \rd W_r^{\ell} \right) f_{q}^{i_2,j_2}(s, u) \rd W_u^q 
\end{align*}	
and applying Fubini's theorem indicate that 
\begin{align*}
\hE \left[\left|(\mathrm{I})_{\ell,q, 1}^{i_1,j_1,i_2,j_2}\right|^2\right] = \hE \bigg[ \int_0^t \int_0^t D_{1, s}^{\ell,q} D_{1, v}^{\ell,q} \rd v \rd s \bigg] = 2 \int_0^t \int_0^s \hE [ D_{1, s}^{\ell,q} D_{1, v}^{\ell,q}] \rd v \rd s. 
\end{align*}
For the integrand of the right integral, using \eqref{eq.ItoProduct Rule} and Fubini's theorem yields 
\begin{align} 
&\ \hE \big[ D_{1, s}^{\ell,q} D_{1, v}^{\ell,q} \big] = \hE \big[ \hE[D_{1, s}^{\ell,q} D_{1, v}^{\ell,q}| \cF_{\underline{v}}] \big] \nonumber\\
&= n^{4\alpha-2} \hE \left[\int_0^{\underline{v}} \left( \int_0^u f_{\ell}^{i_1,j_1}(s, r) \rd W_r^{\ell} \right) f_{q}^{i_2,j_2}(s, u) \rd W_u^q\int_0^{\underline{v}} \left( \int_0^u f_{\ell}^{i_1,j_1}(v, r) \rd W_r^{\ell} \right) f_{q}^{i_2,j_2}(v, u) \rd W_u^q\right]\nonumber\\
&= n^{4\alpha-2} \hE \bigg[ \int_0^{\underline{v}} \bigg( \int_0^u f_{\ell}^{i_1,j_1}(s, r) \rd W_r^{\ell} \bigg) \bigg( \int_0^u f_{\ell}^{i_1,j_1}(v, r) \rd W_r^{\ell} \bigg) f_{q}^{i_2,j_2}(s, u) f_{q}^{i_2,j_2}(v, u) \rd u \bigg] \nonumber \\ 
&= n^{4\alpha-2} \int_0^{\underline{v}} \big( K^{i_2}_{j_2}(s-u)- K^{i_2}_{j_2}(\underline{s}-\underline{u}) \big) \big( K^{i_2}_{j_2}(v-u)- K^{i_2}_{j_2}(\underline{v}-\underline{u}) \big)\tilde{f}(u) \rd u \label{eq.E[DD]}
\end{align}
with 
\begin{align*}
\tilde{f}(u) := \hE \left[ \left| \sigma_{q}^{j_{2}} \left(\widehat{X}_{\underline{u}} \right) \right|^{2} \left( \int_{0}^{u} f_{\ell}^{i_1,j_1}(s, r) \rd W_{r}^{\ell} \right) \left( \int_{0}^{u} f_{\ell}^{i_1,j_1}(v, r) \rd W_{r}^{\ell} \right) \right]. 
\end{align*} 
For $p \geq 2$, $0 \leq u \leq \underline{s}$ with $s \leq T$, there is a constant $C>0$ independent of $n$, $u$ and $s$ such that 
\begin{align} 
\bigg\| \int_{0}^{u} f_{\ell}^{i_1,j_1}(s, r) \rd W_{r}^{\ell} \bigg\|_{L^{p}} 
\leq C \left\| \int_{0}^{\underline{s}} \left| K^{k}_{j}(s-r)- K^{k}_{j} \left({\underline{s}}-{\underline{r}} \right) \right|^{2} \left| \sigma_{\ell}^{j} (\widehat{X}_{\underline{r}}) \right|^{2} \rd r \right\|_{L^{p / 2}}^{\frac{1}{2}} 
\leq C n^{\frac{1}{2}-\alpha} \label{eq.|f|}. 
\end{align}
Then using H\"older's inequality yields 
\begin{align*}
|\tilde{f}(u)| 
\leq \left\| \left| \sigma_{q}^{j_{2}} (\widehat{X}_{\underline{u}} ) \right|^{2} \right\|_{L^{2}} \left\| \int_{0}^{u} f_{\ell}^{i_1,j_1}(s,r) \rd W_{r}^{\ell} \right\|_{L^{4}} \left\| \int_{0}^{u} f_{\ell}^{i_1,j_1}(v, r) \rd W_{r}^{\ell} \right\|_{L^{4}} 
\leq C n^{1-2\alpha},
\end{align*}
which together with \eqref{eq.E[DD]} implies that 
\begin{align*}
\left| \hE [ D_{1, s}^{\ell,q} D_{1, v}^{\ell,q} ] \right| 
\leq C n^{2\alpha-1} \int_{0}^{\underline{v}} \left| \big( K^{i_2}_{j_2}(s-u)- K^{i_2}_{j_2}(\underline{s}-\underline{u}) \big) \big( K^{i_2}_{j_2}(v-u)- K^{i_2}_{j_2}(\underline{v}-\underline{u}) \big) \right| \rd u.
\end{align*}
Then, Lemma \ref{lm.cB_n} leads to $\hE[ D_{1, s}^{\ell,q} D_{1, v}^{\ell,q}] \xrightarrow[n \to \infty ]{} 0$. Therefore, it follows from the bounded convergence theorem (BCT) w.r.t.\ $ \rd v\otimes \rd s$ that $(\mathrm{I})_{\ell,q, 1}^{i_1,j_1,i_2,j_2} \xrightarrow[n \to \infty ]{L^2(\Omega)} 0$. Similarly, it holds that $(\mathrm{I})_{\ell,q, 2}^{i_1,j_1,i_2,j_2} \xrightarrow[n \to \infty ]{L^2(\Omega)} 0$.

Next, we handle $(\mathrm{I})_{\ell,q, 3}^{i_1,j_1,i_2,j_2}$. For the case $\ell \neq q$,
\begin{align} \label{eq.I_ell,q,3=0}
(\mathrm{I})_{\ell,q, 3}^{i_1,j_1,i_2,j_2} = n^{2\alpha-1} \int_0^t \int_0^{\underline{s}} f_{\ell}^{i_1,j_1}(s, u) f_{q}^{i_2,j_2}(s, u) \rd \< W^{\ell}, W^q \>_u \rd s=0.
\end{align}
For the case $\ell = q$, applying the change of variables $y = \lfloor ns \rfloor - nu$ yields
\small{
\begin{align*}
(\mathrm{I})_{\ell, \ell, 3}^{i_1,j_1,i_2,j_2} 
&= n^{2\alpha-1} \int_{0}^{t} \int_0^{\underline{s}} f_{\ell}^{i_1,j_1}(s, u) f_{\ell}^{i_2,j_2}(s, u) \rd u \rd s \\ 
&= n^{2\alpha-1} \int_0^t \int_0^{\underline{s}} \left( K^{i_1}_{j_1}(s-u)- K^{i_1}_{j_1} (\underline{s}-\underline{u}) \right) \left( K^{i_2}_{j_2}(s-u)- K^{i_2}_{j_2} (\underline{s}-\underline{u}) \right) \sigma_{\ell}^{j_1}(\widehat{X}_{\underline{u}}) \sigma_{\ell}^{j_2}(\widehat{X}_{\underline{u}}) \rd u \rd s \\ 
&= \int_0^t \int_0^{ \lfloor ns \rfloor} (\cJ_n)^{i_1,i_2}_{j_1,j_2}(s,y) \sigma_{\ell}^{j_1}(\widehat{X}_{\frac{ \lfloor ns \rfloor + \lfloor-y \rfloor}{n}}) \sigma_{\ell}^{j_2}(\widehat{X}_{\frac{ \lfloor ns \rfloor + \lfloor -y \rfloor}{n}}) \rd y \rd s, 
\end{align*}}
where $(\cJ_n)^{i_1,i_2}_{j_1,j_2}(s,y)$ is defined in \eqref{eq.F^n(s,y)}. In view of H\"older's inequality, \eqref{eq.lips}, \eqref{eq.momentBounded} and \eqref{X-hatX_Lp}, it follows that
\begin{align} \label{eq.sigma^2-sigma^2}
&\ \big\| \sigma_{\ell}^{j_{1}}(\widehat{X}_{\frac{ \lfloor ns \rfloor + \lfloor-y \rfloor}{n}}) \sigma_{\ell}^{j_{2}}(\widehat{X}_{\frac{ \lfloor ns \rfloor+ \lfloor-y \rfloor}{n}})- \sigma_{\ell}^{j_1}(X_s) \sigma_{\ell}^{j_2}(X_s) \big\|_{L^{2}} \nonumber\\ 
&\leq \big\| \sigma_{\ell}^{j_1}(\widehat{X}_{\frac{ \lfloor ns \rfloor + \lfloor-y \rfloor}{n}}) \big(\sigma_{\ell}^{j_2}(\widehat{X}_{\frac{ \lfloor ns \rfloor + \lfloor-y \rfloor}{n}})- \sigma_{\ell}^{j_2}(X_s)\big)\big\|_{L^2} 
+ \big\| \big(\sigma_{\ell}^{j_1}(\widehat{X}_{\frac{ \lfloor ns \rfloor + \lfloor-y \rfloor}{n}})- \sigma_{\ell}^{j_1}(X_s)\big) \sigma_{\ell}^{j_2}(X_s) \big\|_{L^2}\nonumber \\ 
&\leq \big\| \sigma_{\ell}^{j_1}(\widehat{X}_{\frac{ \lfloor ns \rfloor + \lfloor-y \rfloor}{n}}) \big\|_{L^4} \big\| \sigma_{\ell}^{j_2}(\widehat{X}_{\frac{ \lfloor ns \rfloor + \lfloor-y \rfloor}{n}})
- \sigma_{\ell}^{j_2}(X_s) \big\|_{L^4} 
+ \big\| \sigma_{\ell}^{j_2}(X_s) \big\|_{L^4} \big\| \sigma_{\ell}^{j_1}(\widehat{X}_{\frac{ \lfloor ns \rfloor + \lfloor-y \rfloor}{n}})- \sigma_{\ell}^{j_1}(X_s) \big\|_{L^4} \nonumber \\ 
&\leq C \Big( \big|\tfrac{ \lfloor ns \rfloor + \lfloor-y \rfloor}{n} - s \big|^{\alpha-\frac{1}{2}} + n^{\frac{1}{2}-\alpha} \Big) 
\leq C \big( (y+2)^{\alpha-\tfrac{1}{2}} + 1 \big) n^{\frac{1}{2}-\alpha} \xrightarrow[n \to \infty]{} 0. 
\end{align}
Moreover, it follows from the linear growth of $\sigma$, \eqref{eq.momentBounded} and \eqref{eq.cJn<Psi} that
\begin{align*}
\int_0^t \int_0^{ \lfloor ns \rfloor} \big|(\cJ_n)^{i_1,i_2}_{j_1,j_2}(s,y)\big| \big\| \sigma_{\ell}^{j_{1}}(\widehat{X}_{\frac{ \lfloor ns \rfloor + \lfloor-y \rfloor}{n}}) \sigma_{\ell}^{j_{2}}(\widehat{X}_{\frac{ \lfloor ns \rfloor+ \lfloor-y \rfloor}{n}})- \sigma_{\ell}^{j_1}(X_s) \sigma_{\ell}^{j_2}(X_s)\big\|_{L^2} \rd y \rd s 
 \leq C. 
\end{align*}
By \eqref{eq.cJn<Psi} and \eqref{eq.sigma^2-sigma^2}, the DCT w.r.t.\ $ \rd y\otimes \rd s$ implies that 
\begin{align} \label{eq.Fnsigma^2-sigma^2}
 \int_0^t \int_0^{ \lfloor ns \rfloor} (\cJ_n)^{i_1,i_2}_{j_1,j_2}(s,y)\big( \sigma_{\ell}^{j_{1}}(\widehat{X}_{\frac{ \lfloor ns \rfloor + \lfloor-y \rfloor}{n}}) \sigma_{\ell}^{j_{2}}(\widehat{X}_{\frac{ \lfloor ns \rfloor+ \lfloor-y \rfloor}{n}})- \sigma_{\ell}^{j_1}(X_s) \sigma_{\ell}^{j_2}(X_s)\big) \rd y \rd s\xrightarrow[n \to \infty]{L^2(\Omega)} 0, 
\end{align}
which together with Lemma \ref{lm.cJ_n} indicates that $(\mathrm{I})_{\ell, \ell, 3}^{i_1,j_1,i_2,j_2}$ and 
\begin{align*}
\widetilde{(\mathrm{I})}_{\ell, \ell, 3}^{i_1,j_1,i_2,j_2} 
:= \int_0^t \int_0^{\infty} (\cG_n)^{i_1,i_2}_{j_1,j_2}(s,y) \sigma_{\ell}^{j_1}(X_{s}) \sigma_{\ell}^{j_2}(X_{s}) \rd y \rd s
\end{align*}
have the same limit in $L^2(\Omega)$ as $n \to \infty$. In addition, an application of \cite[Lemma C.2]{FukasawaUgai2023} with \eqref{eq.cJn<Psi} gives
\begin{align*}
\widetilde{(\mathrm{I})}_{\ell, \ell, 3}^{i_1,j_1,i_2,j_2} 
\xrightarrow[n \to \infty]{L^2(\Omega)} 
\begin{cases}
\kappa_{1,1}^{2}(\alpha) \int_0^t \ \sigma_{\ell}^{j_1}(X_{s}) \sigma_{\ell}^{j_2}(X_{s}) \rd s, & i_1 = j_1 \text { and } i_2 = j_2, \\
0, &\text {otherwise}, 
\end{cases}
\end{align*}
where the constant $\kappa_{1,1}^{2}(\alpha)$ is defined in \eqref{eq.kappa_1,1^2}. Recalling \eqref{eq.I_ell,q,3=0} completes the proof of \eqref{eq.(17)}.

\underline{\emph{Proof of \eqref{eq.(18)}}}.
We write 
\begin{align*}
(\mathrm{II})_{\ell,q}^{i_1,j_1,i_2,j_2} 
= \int_0^t \sigma_q^{j_2}(\widehat{X}_{\underline{s}}) D_{2, s}^{\ell,q} \rd s 
\quad \mbox{with} \quad 
D_{2, s}^{\ell,q} 
:= n^{2\alpha-1} \left( \int_0^{\underline{s}} f_{\ell}^{i_1,j_1}(s, u) \rd W_u^{\ell} \right) \left( \int_{\underline{s}}^s K^{i_2}_{j_2}(s-u) \rd W_u^q \right).
\end{align*} 
Then, using Fubini's theorem yields 
\begin{align*}
&\ \hE \left[ \left| (\mathrm{II})_{\ell,q}^{i_1,j_1,i_2,j_2} \right|^2 \right] 
= \hE \left[ \int_0^t \int_0^t \sigma_q^{j_2}(\widehat{X}_{\underline{s}}) D_{2, s}^{\ell,q} \sigma_{q}^{j_2}(\widehat{X}_{\underline{v}}) D_{2, v}^{\ell,q} \rd v \rd s \right] \\ 
&= 2 \int_0^t \int_0^{\underline{s}} \hE [ \sigma_q^{j_2}(\widehat{X}_{\underline{s}}) D_{2, s}^{\ell,q} \sigma_q^{j_2}(\widehat{X}_{\underline{v}}) D_{2, v}^{\ell,q}] \rd v \rd s + 2 \int_0^t \int_{\underline{s}}^s \hE [ \sigma_q^{j_2}(\widehat{X}_{\underline{s}}) D_{2, s}^{\ell,q} \sigma_q^{j_2}(\widehat{X}_{\underline{v}}) D_{2, v}^{\ell,q}] \rd v \rd s \\ 
&=: 2(\mathrm{II})_{\ell,q,1}^{i_1,j_1,i_2,j_2} + 2(\mathrm{II})_{\ell,q,2}^{i_1,j_1,i_2,j_2}.
\end{align*}
On one hand, it follows from 
\begin{align*}
\hE [ D_{2, s}^{\ell,q} | \cF_{\underline{s}}] = \int_0^{\underline{s}} f_{\ell}^{i_1,j_1}(s, u) \rd W_u^{\ell} \, \hE \left[ \int_{\underline{s}}^s K^{i_2}_{j_2}(s-u) \rd W_u^q \big| \cF_{\underline{s}} \right] = 0 
\end{align*}
that for all $n \in \hN_+$, 
\begin{align*}
(\mathrm{II})_{\ell,q,1}^{i_1,j_1,i_2,j_2} 
= \int_0^t \int_{0}^{\underline{s}} \hE \left[ \sigma_q^{j_2}(\widehat{X}_{\underline{s}}) \sigma_q^{j_2}(\widehat{X}_{\underline{v}}) D_{2,v}^{\ell,q} \hE [ D_{2, s}^{\ell,q} 
| \cF_{\underline{s}}] \right] \rd v \rd s 
= 0. 
\end{align*} 
On the other hand, using H\"older's inequality, \eqref{eq.|f|}, the BDG inequality and \eqref{eq.|K|} shows that there exists $C>0$ independent of $n$ such that for all $s \in [0,T]$, 
\begin{align*}
 \big\| D_{2, s}^{\ell,q} \big\|_{L^4} 
&\leq n^{2\alpha-1} \left\| \int_0^{\underline{s}} f_{\ell}^{i_1,j_1}(s, u) \rd W_u^{\ell} \right\|_{L^8} \left\| \int_{\underline{s}}^s K^{i_2}_{j_2}(s-u) \rd W_u^q \right\|_{L^8} \\ 
&\leq C n^{\alpha-\frac{1}{2}} \left( \int_{\underline{s}}^s \left| K^{i_2}_{j_2}(s-u) \right|^2 \rd u \right)^{\frac12} \leq C, 
\end{align*} 
which together with the linear growth of $\sigma$ and \eqref{eq.momentBounded} indicates 
\begin{align*}
\left| \hE \left[ \sigma_{q}^{{j_{2}}}(\widehat{X}_{\underline{s}}) D_{2, s}^{\ell,q} \sigma_{q}^{{j_{2}}}(\widehat{X}_{{\underline{v}}}) D_{2, v}^{\ell,q} \right] \right|
\leq \big\| \sigma_{q}^{{j_{2}}}(\widehat{X}_{\underline{s}}) \big\|_{L^{4}} \big\| D_{2, s}^{\ell,q} \big\|_{L^{4}} \big\| \sigma_{q}^{{j_{2}}}(\widehat{X}_{{\underline{v}}}) \big\|_{L^{4}} \big\| D_{2, v}^{\ell,q} \big\|_{L^{4}} 
\leq C. 
\end{align*} 
Then, it follows from the BCT that 
\begin{align*}
\lim_{n \to \infty} (\mathrm{II})_{\ell,q,2}^{i_1,j_1,i_2,j_2} = \int_0^t \int_0^t \lim_{n \to \infty}1_{(\underline{s}, s)}(v) \hE [ \sigma_q^{j_2}(\widehat{X}_{\underline{s}}) D_{2, s}^{\ell,q} \sigma_q^{j_2}(\widehat{X}_{\underline{v}}) D_{2, v}^{\ell,q}] \rd v \rd s = 0. 
\end{align*} 
Thus, one has $\lim_{n \to \infty}\mathbb{E}[|(\mathrm{II})_{\ell,q}^{i_1,j_1,i_2,j_2}|^2] = 0$, thereby verifying \eqref{eq.(18)}. We omit the proof of \eqref{eq.(19)} since it follows from the same arguments used in the proof of \eqref{eq.(18)}.

\underline{\emph{Proof of \eqref{eq.(20)}}}. 
In view of \eqref{eq.ItoProduct Rule}, one has 
\begin{align} \label{eq.iv_lq}
(\mathrm{IV})_{\ell,q}^{i_1,j_1,i_2,j_2} 
= (\mathrm{IV})_{\ell,q,1}^{i_1,j_1,i_2,j_2} + (\mathrm{IV})_{\ell,q,2}^{i_1,j_1,i_2,j_2} + (\mathrm{IV})_{\ell,q,3}^{i_1,j_1,i_2,j_2}
\end{align}
with 
\begin{align*} 
&(\mathrm{IV})_{\ell,q,1}^{i_1,j_1,i_2,j_2} 
:= n^{2\alpha-1} \int_0^t \sigma_{\ell}^{j_1}(\widehat{X}_{\underline{s}}) \sigma_q^{j_2}(\widehat{X}_{\underline{s}}) \bigg[ \int_{\underline{s}}^s \biggl( \int_{\underline{s}}^u K^{i_1}_{j_1}(s-r) \rd W_r^{\ell} \biggr) K^{i_2}_{j_2}(s-u) \rd W_u^q \bigg] \rd s, \\ 
&(\mathrm{IV})_{\ell,q,2}^{i_1,j_1,i_2,j_2} 
:= n^{2\alpha-1} \int_0^t \sigma_{\ell}^{j_1}(\widehat{X}_{\underline{s}}) \sigma_q^{j_2}(\widehat{X}_{\underline{s}}) \bigg[ \int_{\underline{s}}^s \left( \int_{\underline{s}}^u K^{i_2}_{j_2}(s-r) \rd W_r^q \right) K^{i_1}_{j_1}(s-u) \rd W_u^{\ell} \bigg] \rd s, \\
&(\mathrm{IV})_{\ell,q,3}^{i_1,j_1,i_2,j_2} 
:= n^{2\alpha-1} \int_0^t \sigma_{\ell}^{j_1}(\widehat{X}_{\underline{s}}) \sigma_q^{j_2}(\widehat{X}_{\underline{s}}) \bigg[ \int_{\underline{s}}^s K^{i_1}_{j_1}(s-u) K^{i_2}_{j_2}(s-u) \rd \< W^{\ell}, W^q \>_u \bigg] \rd s.
\end{align*}
When $\ell \neq q$, $(\mathrm{IV})_{\ell,q,3}^{i_1,j_1,i_2,j_2} = 0$. While $\ell = q$, it holds that 
\begin{align*}
(\mathrm{IV})_{\ell,\ell,3}^{i_1,j_1,i_2,j_2}
&= n^{2\alpha-1} \int_0^t \sigma_{\ell}^{j_1}(\widehat{X}_{\underline{s}}) \sigma_{\ell}^{j_2}(\widehat{X}_{\underline{s}}) \bigg[ \int_{\underline{s}}^s K^{i_1}_{j_1}(s-u) K^{i_2}_{j_2}(s-u) \rd u \bigg] \rd s \\
&= \int_0^t \sigma_{\ell}^{j_1}(\widehat{X}_{\underline{s}}) \sigma_{\ell}^{j_2}(\widehat{X}_{\underline{s}}) \bigg[ \int_0^{ns- \lfloor ns \rfloor} \cK^{2}(v) \cE^{i_1}_{j_1}\left(\frac{v}{n}\right) \cE^{i_2}_{j_2}\left(\frac{v}{n}\right) \rd v \bigg] \rd s, 
\end{align*}
where the last step used $K(\cdot) = \cK(\cdot) \cE(\cdot)$, \eqref{eq.def:k&e} and 
\begin{align*}
\int_{\underline{s}}^s \cK^{2}(s-u) \cE^{i_1}_{j_1}(s-u) \cE^{i_2}_{j_2}(s-u) \rd u 
= \int_0^{s-\underline{s}} \cK^{2}(r) \cE^{i_1}_{j_1}(r) \cE^{i_2}_{j_2}(r) \rd r 
= n^{1-2\alpha} \int_0^{ns- \lfloor ns \rfloor} \cK^{2}(v) \cE^{i_1}_{j_1}\left(\frac{v}{n}\right) \cE^{i_2}_{j_2}\left(\frac{v}{n}\right) \rd v. 
\end{align*}

We define the function $\cY^{i_1,i_2}_{j_1,j_2}(\cdot) := \delta_{i_1,j_1} \delta_{i_2,j_2} \frac{1}{\Gamma^{2}(\alpha)} \cK^{2}(\cdot)$ (where $\delta$ is the Kronecker delta) and deduce that 
\begin{align*}
&\ \left\| \int_0^t \sigma_{\ell}^{j_1}(\widehat{X}_{\underline{s}}) \sigma_{\ell}^{j_2}(\widehat{X}_{\underline{s}}) \int_0^{ns- \lfloor ns \rfloor} \left(\cK^{2}(v) \cE^{i_1}_{j_1}\left(\frac{v}{n}\right) \cE^{i_2}_{j_2}\left(\frac{v}{n}\right) -\cY^{i_1,i_2}_{j_1,j_2}(v) \right) \rd v \rd s \right\|_{L^2} \\ 
&\leq \int_0^t \left\| \sigma_{\ell}^{j_1}(\widehat{X}_{\underline{s}}) \sigma_{\ell}^{j_2}(\widehat{X}_{\underline{s}}) \right\|_{L^2} \int_0^{1} 1_{(0, ns- \lfloor ns \rfloor)}(v) \left|\cK^{2}(v) \cE^{i_1}_{j_1}\left(\frac{v}{n}\right) \cE^{i_2}_{j_2}\left(\frac{v}{n}\right) - \cY^{i_1,i_2}_{j_1,j_2}(v) \right| \rd v \rd s\\
&\leq C \int_0^t \int_0^{1} 1_{(0, ns- \lfloor ns \rfloor)}(v) \left|\cK^{2}(v) \cE^{i_1}_{j_1}\left(\frac{v}{n}\right) \cE^{i_2}_{j_2}\left(\frac{v}{n}\right) - \cY^{i_1,i_2}_{j_1,j_2}(v) \right| \rd v \rd s. 
\end{align*} 
By \eqref{eq.cU_Holder}, it holds that for all $v \in (0,T]$, 
\begin{align*}
\left|\cK^{2}(v) \cE^{i_1}_{j_1}\left(\frac{v}{n}\right) \cE^{i_2}_{j_2}\left(\frac{v}{n}\right) - \cY^{i_1,i_2}_{j_1,j_2}(v) \right|
\leq C \cK^{2}(v),
\end{align*} 
where $C>0$ is independent of $n$ and $v$. Note that $\int_0^t \int_0^{1} 1_{(0, ns- \lfloor ns \rfloor)}(v) \cK^{2}(v) \rd v \rd s\leq C$. And it follows from \eqref{eq.limeij} that 
\begin{align*}
\lim_{n \to \infty}\left|\cK^{2}(v) \cE^{i_1}_{j_1}\left(\frac{v}{n}\right) \cE^{i_2}_{j_2}\left(\frac{v}{n}\right) - \cY^{i_1,i_2}_{j_1,j_2}(v) \right|=0,
\end{align*}
which implies 
\begin{align*}
\lim_{n \to \infty}1_{{(0, ns- \lfloor ns \rfloor)}}(v)\, \left|\cK^{2}(v) \cE^{i_1}_{j_1}\left(\frac{v}{n}\right) \cE^{i_2}_{j_2}\left(\frac{v}{n}\right) - \cY^{i_1,i_2}_{j_1,j_2}(v) \right| = 0. 
\end{align*}
Thus, applying the DCT yields 
\begin{align*}
\lim_{n \to \infty} \int_0^t \int_0^{1} 1_{(0, ns- \lfloor ns \rfloor)}(v) \left|\cK^{2}(v) \cE^{i_1}_{j_1}\left(\frac{v}{n}\right) \cE^{i_2}_{j_2}\left(\frac{v}{n}\right) - \cY^{i_1,i_2}_{j_1,j_2}(v) \right| \rd v \rd s=0, 
\end{align*}
which indicates 
\begin{align} \label{eq.k_to_y}
\int_0^t \sigma_{\ell}^{j_1}(\widehat{X}_{\underline{s}}) \sigma_{\ell}^{j_2}(\widehat{X}_{\underline{s}}) \int_0^{1} 1_{(0, ns- \lfloor ns \rfloor)}(v) \left(\cK^{2}(v) \cE^{i_1}_{j_1}\left(\frac{v}{n}\right) \cE^{i_2}_{j_2}\left(\frac{v}{n}\right) -\cY^{i_1,i_2}_{j_1,j_2}(v) \right) \rd v \rd s \xrightarrow[n \to \infty]{L^2(\Omega)} 0. 
\end{align}
In light of the relation 
\begin{align*} 
\hE \left[ \int_0^T \left| \sigma_{\ell}^{i_1}(\widehat{X}_{\underline{s}}) \sigma_{\ell}^{i_2}(\widehat{X}_{\underline{s}})- \sigma_{\ell}^{i_1}(X_s) \sigma_{\ell}^{i_2}(X_s) \right|^2 \rd s \right] \xrightarrow[n \to \infty]{} 0, 
\end{align*}
one can also use \cite[Lemma C.2]{FukasawaUgai2023} to obtain 
\begin{align*}
&\ \int_0^t \sigma_{\ell}^{j_1}(\widehat{X}_{\underline{s}}) \sigma_{\ell}^{j_2}(\widehat{X}_{\underline{s}}) \int_0^{1} 1_{(0, ns- \lfloor ns \rfloor)}(v) \cY^{i_1,i_2}_{j_1,j_2}(v) \rd v \rd s \\
&\xrightarrow[n \to \infty]{L^{2}(\Omega)}
\begin{cases}
 \frac{1} {\Gamma^{2}(\alpha)2\alpha(2\alpha-1)} \int_0^t \sigma_{\ell}^{i_1}(X_{s}) \sigma_{\ell}^{i_2}(X_{s}) \rd s, &i_1=j_1 \text{ and } i_2=j_2,\\0, &\text{otherwise},
\end{cases}
\end{align*}
which together with \eqref{eq.k_to_y} shows 
\begin{align*}
&\int_0^t \sigma_{\ell}^{j_1}(\widehat{X}_{\underline{s}}) \sigma_{\ell}^{j_2}(\widehat{X}_{\underline{s}}) \int_0^{ns- \lfloor ns \rfloor} \cK^{2}(v) \cE^{i_1}_{j_1}\left(\frac{v}{n}\right) \cE^{i_2}_{j_2}\left(\frac{v}{n}\right) \rd v \rd s \notag\\
&\xrightarrow[n \to \infty]{L^{2}(\Omega)}
\begin{cases}
 \frac{1} {\Gamma^{2}(\alpha)2\alpha(2\alpha-1)} \int_0^t \sigma_{\ell}^{i_1}(X_{s}) \sigma_{\ell}^{i_2}(X_{s}) \rd s, &i_1=j_1 \text{ and } i_2=j_2,\\0, &\text{otherwise}. 
 \end{cases} 
\end{align*}
Therefore, one can conclude that 
\begin{align} \label{eq.limit_IV_3}
(\mathrm{IV})_{\ell,q,3}^{i_1,j_1,i_2,j_2} 
\xrightarrow[n \to \infty]{L^{2}(\Omega)}
\begin{cases}
 \frac{1} {\Gamma^{2}(\alpha)2\alpha(2\alpha-1)} \int_0^t \sigma_{\ell}^{i_1}(X_{s}) \sigma_{\ell}^{i_2}(X_{s}) \rd s, &\ell=q,\, i_1=j_1, \text{ and } i_2=j_2,\\
 0, &\text{otherwise}. 
 \end{cases} 
\end{align}

We turn to the first term on the right-hand side of \eqref{eq.iv_lq}. For brevity, we write 
\begin{align*}
(\mathrm{IV})_{\ell,q,1}^{i_1,j_1,i_2,j_2} = \int_0^t E^{j_1,j_2}_{\ell,q,s} D_{4,s}^{\ell,q} \rd s
\end{align*} 
with $E^{j_1,j_2}_{\ell,q,s} := \sigma_{\ell}^{j_1}(\widehat{X}_{\underline{s}}) \sigma_q^{j_2}(\widehat{X}_{\underline{s}})$, $D_{4,s}^{\ell,q} 
:= n^{2\alpha-1} \int_{\underline{s}}^s g_{\ell}(s, u) \rd W_u^q$, and $g_{\ell}(s, u) := K^{i_2}_{j_2}(s-u) \int_{\underline{s}}^u K^{i_1}_{j_1}(s-r) \rd W_r^{\ell}$. Using Fubini's theorem yields 
\begin{align} 
&\ \hE \left[ \left| (\mathrm{IV})_{\ell,q,1}^{i_1,j_1,i_2,j_2} \right|^{2}\right] 
= \hE \bigg[ \int_0^t \int_0^t E^{j_1,j_2}_{\ell,q,s} D_{4,s}^{\ell,q} E^{j_1,j_2}_{\ell,q,v} D_{4,v}^{\ell,q} \rd v \rd s \bigg] \notag \\ 
&= 2 \int_0^t \int_0^{\underline{s}} \hE \left[ E^{j_1,j_2}_{\ell,q,s} D_{4,s}^{\ell,q} E^{j_1,j_2}_{\ell,q,v} D_{4,v}^{\ell,q} \right] \rd v \rd s + 2 \int_0^t \int_{\underline{s}}^s \hE \left[ E^{j_1,j_2}_{\ell,q,s} D_{4,s}^{\ell,q} E^{j_1,j_2}_{\ell,q,v} D_{4,v}^{\ell,q} \right] \rd v \rd s \notag \\ 
&=: 2 (\mathrm{IV})_{\ell,q,1,1}^{i_1,j_1,i_2,j_2} + 2 (\mathrm{IV})_{\ell,q,1,2}^{i_1,j_1,i_2,j_2}. \label{eq.IV_1_split}
\end{align}
On one hand, using the facts that $E^{j_1,j_2}_{\ell,q,s}$ is $\cF_{\underline{s}}$-measurable and $\hE \big[D_{4,s}^{\ell,q}\, \big|\, \cF_{\underline{s}} \big] = \hE \big[D_{4,s}^{\ell,q} \big] = 0$ gets 
\begin{align} \label{eq.limit_IV_11}
(\mathrm{IV})_{\ell,q,1,1}^{i_1,j_1,i_2,j_2} 
= \int_0^t \int_0^{\underline{s}} \hE \left[ E^{j_1,j_2}_{\ell,q,s} E^{j_1,j_2}_{\ell,q,v} D_{4,v}^{\ell,q} \hE \big[D_{4,s}^{\ell,q}\, \big|\, \cF_{\underline{s}} \big] \right] \rd v \rd s = 0.
\end{align}
On the other hand, based on the fact that $\underline{v} = \underline{s}$ for $v \in (\underline{s}, s)$, one has 
\begin{align*}
(\mathrm{IV})_{\ell,q,1,2}^{i_1,j_1,i_2,j_2} 
&= \int_{0}^{t} \int_{0}^{t} 1_{(\underline{s}, s)}(v) \hE \left[ \big( E^{j_1,j_2}_{\ell,q,s} \big)^2 D_{4,s}^{\ell,q} D_{4,v}^{\ell,q} \right] \rd v \rd s \\
&= \int_{0}^{t} \int_{0}^{t} 1_{(\underline{s}, s)}(v) \hE \left[ \big( E^{j_1,j_2}_{\ell,q,s} \big)^2 \hE \big[ D_{4,s}^{\ell,q} D_{4,v}^{\ell,q} \, \big|\, \cF_{\underline{v}} \big] \right] \rd v \rd s. 
\end{align*} 
Here, by \eqref{eq.ItoProduct Rule}, the conditional expectation satisfies 
\begin{align*}
\hE \big[ D_{4,s}^{\ell,q} D_{4,v}^{\ell,q} \, \big|\, \cF_{\underline{v}} \big] 
&= n^{4\alpha-2} \hE \left[ \int_{\underline{s}}^{s} g_{\ell}(s, u) \rd W_{u}^{q} \int_{\underline{s}}^{s} 1_{(\underline{v}, v)}(u) g_{\ell}(v, u) \rd W_{u}^{q} \, \big|\, \cF_{\underline{v}} \right] \\
&= n^{4\alpha-2} \int_{\underline{v}}^{v} \hE \big[ g_{\ell}(s, u) g_{\ell}(v, u) \, \big|\, \cF_{\underline{v}} \big] \rd u, 
\end{align*} 
where 
\begin{align*}
\hE \big[ g_{\ell}(s, u) g_{\ell}(v, u) \, \big|\,\cF_{\underline{v}} \big] 
&= K^{i_2}_{j_2}(s-u) K^{i_2}_{j_2}(v-u) \hE \left[ \left( \int_{\underline{v}}^{u} K^{i_1}_{j_1}(s-r) \rd W_{r}^{\ell} \right) \left( \int_{\underline{v}}^{u} K^{i_1}_{j_1}(v-r) \rd W_{r}^{\ell} \right) \big|\cF_{\underline{v}} \right] \\ 
&= K^{i_2}_{j_2}(s-u) K^{i_2}_{j_2}(v-u) \int_{\underline{v}}^{u} K^{i_1}_{j_1}(s-r) K^{i_1}_{j_1}(v-r) \rd r.
\end{align*}
For $s \in \left[0, T \right]$ and $v \in (\underline{s}, s)$, using \eqref{eq.cU_Holder} and \eqref{eq.|K|} deduces 
\begin{align*}
\big| \hE \big[ D_{4,s}^{\ell,q} D_{4,v}^{\ell,q} \, \big|\, \cF_{\underline{v}} \big] \big| 
&= n^{4\alpha-2} \left| \int_{\underline{v}}^{v} K^{i_2}_{j_2}(s-u) K^{i_2}_{j_2}(v-u) \int_{\underline{v}}^{u} K^{i_1}_{j_1}(s-r) K^{i_1}_{j_1}(v-r) \rd r \rd u \right| \\ 
&\leq n^{4\alpha-2} \int_{\underline{v}}^{v} \left| K^{i_2}_{j_2}(s-u) K^{i_2}_{j_2}(v-u) \right| \int_{\underline{v}}^{u} \left| K^{i_1}_{j_1}(s-r) K^{i_1}_{j_1}(v-r) \right| \rd r \rd u \\ 
&\leq C n^{4\alpha-2} \int_{\underline{v}}^{v}\cK(s-u)\cK(v-u) \int_{\underline{v}}^{u} \cK(s-r)\cK(v-r) \rd r \rd u \\ 
&\leq C n^{4\alpha-2} \int_{\underline{v}}^{v}\cK^{2}(v-u) \int_{\underline{v}}^{u} \cK^{2}(v-r) \rd r \rd u \\ 
&\leq C \left(n^{2\alpha-1} \int_{\underline{v}}^{v} \cK^{2}(v-u) \rd u \right)^2 \leq C, 
\end{align*}
which implies 
\begin{align*}
\left| 1_{(\underline{s}, s)}(v) \hE \left[ \big( E^{j_1,j_2}_{\ell,q,s} \big)^2 D_{4,s}^{\ell,q} D_{4,v}^{\ell,q} \right] \right| 
\leq C \hE \left[ \big( E^{j_1,j_2}_{\ell,q,s} \big)^2 \right] 
\leq C
\end{align*}
with $C$ being independent of $n$, $s$ and $v$. Therefore, by the BCT w.r.t.\ $ \rd v \otimes \rd s$, one gets 
\begin{align*}
\left| (\mathrm{IV})_{\ell,q,1,2}^{i_1,j_1,i_2,j_2} \right|
\leq \int_{0}^{t} \int_{0}^{t} 1_{(\underline{s}, s)}(v) \left| \hE \left[ \big( E^{j_1,j_2}_{\ell,q,s} \big)^2 D_{4,s}^{\ell,q} D_{4,v}^{\ell,q} \right] \right| \rd v \rd s 
\xrightarrow[n \to \infty]{} 0,
\end{align*}
which together with \eqref{eq.IV_1_split} and \eqref{eq.limit_IV_11} indicates $(\mathrm{IV})_{\ell,q,1}^{i_1,j_1,i_2,j_2} \xrightarrow[n \to \infty]{L^{2}(\Omega)} 0$. Using the symmetry between $(\mathrm{IV})_{\ell,q,1}^{i_1,j_1,i_2,j_2}$ and $(\mathrm{IV})_{\ell,q,2}^{i_1,j_1,i_2,j_2}$, we also have $(\mathrm{IV})_{\ell,q,2}^{i_1,j_1,i_2,j_2} \xrightarrow[n \to \infty]{L^{2}(\Omega)} 0$. In conclusion, recalling \eqref{eq.iv_lq} and \eqref{eq.limit_IV_3} shows that \eqref{eq.(20)} holds. 

Hereto, the proof of Lemma \ref{lm.main} is completed. 
\end{proof}

\subsection{Proof of \eqref{lm.covV(ii)}} 

According to Fubini's theorem, one gets 
\begin{align*}
\mathbb{E}\left[\left|\< \hat{V}^{n,k,\ell}, W^{\ell} \>_t\right|^2\right] 
= \mathbb{E}\left[\left| \int_0^t n^{\alpha-\frac{1}{2}} (\widetilde{X}_s^k - \widehat{X}_{\underline{s}}^k) \rd s \right|^2\right] 
= 2 n^{2\alpha-1} \int_0^{t}\int_0^s \mathbb{E}\left[ (\widetilde{X}^k_s - \widehat{X}^k_{\underline{s}}) (\widetilde{X}^k_v - \widehat{X}^k_{\underline{v}}) \right] \rd v \rd s.
\end{align*}
For the integrand of the right integral, it follows from H\"older's inequality and \eqref{X-hatX_Lp} that 
\begin{align*}
\left| \mathbb{E}\left[ (\widetilde{X}^k_s - \widehat{X}^k_{\underline{s}}) (\widetilde{X}^k_v - \widehat{X}^k_{\underline{v}}) \right] \right| 
\leq C \big\| \widetilde{X}_s - \widehat{X}_{\underline{s}} \big\|_{L^2} \big\| \widetilde{X}_v - \widehat{X}_{\underline{v}} \big\|_{L^2} 
\leq C n^{1-2\alpha}. 
\end{align*}
Thus, in light of the BCT, it remains to show that $n^{2\alpha-1} \mathbb{E} \big[ (\widetilde{X}^k_s - \widehat{X}^k_{\underline{s}}) (\widetilde{X}^k_v - \widehat{X}^k_{\underline{v}}) \big] \rightarrow 0$ as $n \to \infty$, for $v < s$. In fact, it only needs to consider the case $v < \underline{s}$. For brevity, we denote
\begin{align*}
h^{k}_{j}(s, u) := \big( K^{k}_{j}(s-u) - K^{k}_{j}(\underline{s} - \underline{u}) \big) b^{j}(\widehat{X}_{\underline{u}}) 
\quad \mbox{and} \quad 
f_{\ell}^{k,j}(s, u) := \big( K^{k}_{j}(s-u) - K^{k}_{j}(\underline{s} - \underline{u}) \big) \sigma_{\ell}^{j}(\widehat{X}_{\underline{u}}),
\end{align*} 
as well as $\psi^{k,j}_{n,0}(s) := (E_{\alpha}^{k,j} (A s^{\alpha}) -E_{\alpha}^{k,j}(A\underline{s}^{\alpha})) X^{j}_{0}$. Then, using \eqref{eq.MLE} and \eqref{eq.auxiMLE} yields that for $v < \underline{s}$, 
\begin{align*}
&\ \mathbb{E} \big[ (\widetilde{X}^k_s - \widehat{X}^k_{\underline{s}}) (\widetilde{X}^k_v - \widehat{X}^k_{\underline{v}}) \big] 
= \mathbb{E} \big[ (\widetilde{X}^k_v - \widehat{X}^k_{\underline{v}}) \mathbb{E} \big[ (\widetilde{X}^k_s - \widehat{X}^k_{\underline{s}}) \, \big|\, \cF_{\underline{s}} \big] \big] \\ 
&= \sum_{j_1=1}^{d} \hE \left[ (\widetilde{X}^k_v - \widehat{X}^k_{\underline{v}}) \bigg( \psi^{k,j_1}_{n,0}(s) + \int_{0}^{\underline{s}} h^{k}_{j_1}(s, u) \rd u 
+ \int_{\underline{s}}^{s} K^{k}_{j_1}(s-u) b^{j_1}(\widehat{X}_{\underline{u}}) \rd u \bigg) \right] \\ 
&\quad + \sum_{j_1=1}^{d} \sum_{\ell=1}^{m} \hE \left[ (\widetilde{X}^k_v - \widehat{X}^k_{\underline{v}}) \int_{0}^{\underline{s}} f_{\ell}^{k,j_1}(s, u) \rd W_{u}^{\ell} \right] \\ 
&= \sum_{j_1=1}^{d} (\mathrm{I})_{j_1} + \sum_{j_1=1}^{d} \sum_{j_2=1}^{d} \sum_{\ell=1}^{m} (\mathrm{II})_{\ell}^{j_1,j_2} + \sum_{j_1=1}^{d}\sum_{j_2=1}^{d}\sum_{\ell=1}^{m} \sum_{q=1}^{m} (\mathrm{III})_{\ell,q}^{j_1,j_2} + \sum_{j_1=1}^{d}\sum_{j_2=1}^{d}\sum_{\ell=1}^{m} \sum_{q=1}^{m} (\mathrm{IV})_{\ell,q}^{j_1,j_2} 
\end{align*} 
with 
\begin{align*}
& (\mathrm{I})_{j_1} := \hE \big[ (\widetilde{X}^k_v - \widehat{X}^k_{\underline{v}}) \big( \psi^{k,j_1}_{n,0}(s) + \int_{0}^{\underline{s}} h^{k}_{j_1}(s, u) \rd u 
+ \int_{\underline{s}}^{s} K^{k}_{j_1}(s-u) b^{j_1}(\widehat{X}_{\underline{u}}) \rd u \big) \big], \\
& (\mathrm{II})_{\ell}^{j_1,j_2} := \hE \left[ \int_{0}^{\underline{s}} f_{\ell}^{k,j_1}(s, u) \rd W_{u}^{\ell} \big( \psi^{k,j_2}_{n,0}(v) + \int_{0}^{\underline{v}} h^{k}_{j_2}(v, u) \rd u 
+ \int_{\underline{v}}^{v} K^{k}_{j_2}(v-u) b^{j_2}(\widehat{X}_{\underline{u}}) \rd u \big) \right], \\
& (\mathrm{III})_{\ell,q}^{j_1,j_2} := \hE \left[ \int_{0}^{\underline{s}} f_{\ell}^{k,j_1}(s, u) \rd W_{u}^{\ell} \int_{0}^{\underline{v}} f_{q}^{k,j_2}(v, u) \rd W_{u}^{q} \right], \\
& (\mathrm{IV})_{\ell,q}^{j_1,j_2} := \hE \left[ \int_{0}^{\underline{s}} f_{\ell}^{k,j_1}(s, u) \rd W_{u}^{\ell} \int_{\underline{v}}^{v} K^{k}_{j_2}(v-u) \sigma_{q}^{j_2}(\widehat{X}_{\underline{u}}) \rd W_{u}^{q} \right]. 
\end{align*} 

Based on H\"older's inequality, \eqref{X-hatX_Lp} and \eqref{eq.cU_Holder} as well as the linear growth of $b$ and \eqref{eq.momentBounded}, one has 
\begin{align*}
& \hE \big[| (\widetilde{X}^k_v - \widehat{X}^k_{\underline{v}}) \psi^{k,j_1}_{n,0}(s) | \big] 
\leq \| \widetilde{X}^k_v - \widehat{X}^k_{\underline{v}} \|_{L^2} \| \psi^{k,j_1}_{n,0}(s) \|_{L^2}
\leq C n^{\frac{1}{2}-2\alpha}, \\
& \hE \big[| (\widetilde{X}^k_v - \widehat{X}^k_{\underline{v}}) b^{j_1}(\widehat{X}_{\underline{u}}) | \big] 
\leq \| \widetilde{X}^k_v - \widehat{X}^k_{\underline{v}} \|_{L^2} \| b^{j_1}(\widehat{X}_{\underline{u}}) \|_{L^2} 
\leq C n^{\frac{1}{2}-\alpha}, 
\end{align*}
where $C>0$ is independent of $n$, $v$, $s$ and $u$. For $(\mathrm{I})_{j_1}$, using the two estimates above and \eqref{lm.|K-K|H} obtains $| (\mathrm{I})_{j_1} | \leq C n^{\frac{1}{2}-2\alpha}$. Then $|n^{2\alpha-1} (\mathrm{I})_{j_1}| \leq C n^{-\frac{1}{2}}$, thereby $n^{2\alpha-1} (\mathrm{I})_{j_1} \to 0$ as $n \to \infty$. For $(\mathrm{II})_{\ell}^{j_1,j_2}$, $n^{2\alpha-1} (\mathrm{II})_{\ell}^{j_1,j_2} \to 0$ as $n \to \infty$ follows by a similar step. 

For $(\mathrm{III})_{\ell,q}^{j_1,j_2}$, it follows from \eqref{eq.ItoProduct Rule}, the linear growth of $\sigma$ and \eqref{eq.momentBounded} that 
\begin{align*}
\big|(\mathrm{III})_{\ell,q}^{j_1,j_2} \big|
&= \left| \hE \left[ \int_{0}^{\underline{s}}f_{\ell}^{k,j_1}(s, u) \rd W_{u}^{\ell} \int_{0}^{\underline{s}}1_{(0, \underline{v})}(u)f_{q}^{k,j_2}(v, u) \rd W_{u}^{q} \right] \right| \\ 
&\leq \sup_{r\in[0,T]} \hE \left[ \big| \sigma_{\ell}^{j_1}(\widehat{X}_{\underline{r}}) \sigma_{\ell}^{j_2}(\widehat{X}_{\underline{r}}) \big| \right] \int_{0}^{\underline{v}} \big| \big( K^{k}_{j_1}(s-u)- K^{k}_{j_1} (\underline{s}-\underline{u}) \big) \big( K^{k}_{j_2}(v-u)- K^{k}_{j_2} (\underline{v}-\underline{u}) \big) \big| \rd u \\ 
&\leq C \int_{0}^{\underline{v}} \big| \big( K^{k}_{j_1}(s-u)- K^{k}_{j_1} (\underline{s}-\underline{u}) \big) \big( K^{k}_{j_2}(v-u)- K^{k}_{j_2} (\underline{v}-\underline{u}) \big) \big| \rd u,
\end{align*}
which together with Lemma \ref{lm.cB_n} shows that $n^{2\alpha-1} (\mathrm{III})_{\ell,q}^{j_1,j_2} \to 0$ as $n \to \infty$.

For $(\mathrm{IV})_{\ell,q}^{j_1,j_2}$, invoking \eqref{eq.ItoProduct Rule}, \eqref{eq.momentBounded}, $K(\cdot) = \cK(\cdot) \cE(\cdot)$ and \eqref{eq.def:k&e} indicates that for $0 < v < \underline{s} < T$, 
\begin{align*}
\big| (\mathrm{IV})_{\ell,q}^{j_1,j_2} \big| 
&= \left| \hE \left[ \int_{0}^{\underline{s}}f_{\ell}^{k,j_1}(s, u) \rd W_{u}^{\ell} \int_{0}^{\underline{s}}1_{(\underline{v}, v)}(u) K^{k}_{j_2}(v-u) \sigma_{q}^{j_2}(\widehat{X}_{\underline{u}}) \rd W_{u}^{q} \right] \right| \\ 
&= \left| \hE \left[ \int_{\underline{v}}^{v}f_{l}^{k,j_1}(s, u) K^{k}_{j_2}(v-u) \sigma_{\ell}^{j_2}(\widehat{X}_{\underline{u}}) \rd u \right] \right|\\ 
&\leq \sup_{r\in[0,T]}\hE \left[ \big| \sigma_{\ell}^{j_1}(\widehat{X}_{\underline{r}}) \sigma_{\ell}^{j_2}(\widehat{X}_{\underline{r}}) \big| \right] \int_{\underline{v}}^{v} \big| K^{k}_{j_1}(s-u)- K^{k}_{j_1} (\underline{s}-\underline{u}) \big| \big| K^{k}_{j_2}(v-u) \big| \rd u \\ 
&\leq C \int_{\underline{v}}^{v}\cK(s-u) \big| \cE^{k}_{j_1}(s-u)- \cE^{k}_{j_1} (\underline{s}-\underline{u}) \big| \cK(v-u) \rd u \\ 
&\quad + C \int_{\underline{v}}^{v} \big| \cE^{k}_{j_1}(\underline{s}-\underline{u}) \big| \big|\cK(s-u) - \cK(\underline{s}-\underline{u}) \big| \cK(v-u) \rd u \\ 
&=: C (\mathrm{IV})_{\ell,q,1}^{j_1,j_2} + C (\mathrm{IV})_{\ell,q,2}^{j_1,j_2}. 
\end{align*}
By \eqref{eq.cU_Holder}, one has $ \big| \cE^{k}_{j_1}(s-u)- \cE^{k}_{j_1} (\underline{s}-\underline{u}) \big| \leq C n^{-\alpha}$, and hence
 \begin{align*}
(\mathrm{IV})_{\ell,q,1}^{j_1,j_2} 
\leq C n^{-\alpha} \int_{\underline{v}}^{v}\cK^2(v-u) \rd u 
\leq C n^{1-3\alpha}, 
\end{align*}
which implies that $n^{2\alpha-1} (\mathrm{IV})_{\ell,q,1}^{j_1,j_2} \to 0$, as $n \to \infty$. For $(\mathrm{IV})_{\ell,q,2}^{j_1,j_2}$, using \eqref{eq.cU_Holder} again and the change of variables $z = n (v-u)$ yields 
\begin{align*}
(\mathrm{IV})_{\ell,q,2}^{j_1,j_2} 
&\leq C \int_{\underline{v}}^{v} \big|\cK(s-u) - \cK(\underline{s}-\underline{u}) \big| \cK(v-u) \rd u \\
&= C n^{1-2\alpha} \int_{0}^{nv- \lfloor nv \rfloor} \big|\cK( z + ns - nv) - \cK( \lfloor ns \rfloor - \lfloor nv-z \rfloor) \big| \cK(z) \rd z. 
\end{align*}
In view of the fact that $|\cK(z + ns - nv) - \cK( \lfloor ns \rfloor - \lfloor nv-z \rfloor)| \to 0$ as $n \to \infty$ for $0 < v < \underline{s} \leq T$, due to the monotonicity of $\cK(\cdot)$, and that 
\begin{align*}
\int_{0}^1 \big|\cK(z + ns - nv) - \cK( \lfloor ns \rfloor - \lfloor nv-z \rfloor) \big| \cK(z) \rd z 
\leq C \int_{0}^1 \cK^2(z) \rd z 
\leq C, 
\end{align*} 
it follows from the DCT that 
\begin{align*} 
\int_{0}^1 \left|\cK(z + ns - nv) - \cK( \lfloor ns \rfloor - \lfloor nv-z \rfloor) \right|\cK(z) \rd z \xrightarrow[n\to\infty]{} 0, 
\end{align*} 
which implies that $n^{2\alpha-1} (\mathrm{IV})_{\ell,q,2}^{j_1,j_2} \to 0$, as $n \to \infty$. Thus, one can conclude that $n^{2\alpha-1} (\mathrm{IV})_{\ell,q}^{j_1,j_2} \to 0$, as $n \to \infty$. Hereto, the proof of \eqref{lm.covV(ii)} is completed. 
\hfill$\Box$

\section{Proofs of Lemmas \ref{lem.replaceR_n1} and Lemma \ref{lm.R_n1}}
 \label{sec.asym_behav_R_n}

\textit{Proof of Lemma \ref{lem.replaceR_n1}}. Firstly, using \eqref{eq.K-reg}, the linear growth of $b$ and \eqref{eq.momentBounded} yields
\begin{align*}
\big\| C_t^{n} \big\|_{L^{2}} 
\leq n^{\alpha-\frac{1}{2}} \sup_{r\in[0,T]} \big\|b(\widehat{X}_{r}) \big\|_{L^2}\int_{0}^{t} \left|K(t-s) - K(t-\underline{s}) \right| \rd s 
\leq C n^{-\frac12}, 
\end{align*}
which completes the proof of Lemma \ref{lem.replaceR_n1}(i).

Fix $\delta>0$, then it follows from \eqref{eq.def_R_tn1} and \eqref{eq.def_tildR_tn1} that 
\begin{align*}
\widehat{R}_{t}^{n} - \widetilde{R}_{t}^{n} 
&= n^{\alpha-\frac12} \int_0^t \big( K(t-s) - K(t-\underline{s}) \big) \big( \sigma(\widehat{X}_{\underline{s}}) - \sigma(X_t) \big) \rd W_s \\ 
&= n^{\alpha-\frac12} \int_0^{(t-\delta)_+} \big( K(t-s) - K(t-\underline{s}) \big) \big( \sigma(\widehat{X}_{\underline{s}}) - \sigma(X_t) \big) \rd W_s \\ 
&\quad + n^{\alpha-\frac12} \int_{(t-\delta)_+}^t \big( K(t-s) - K(t-\underline{s}) \big) \big( \sigma(\widehat{X}_{\underline{s}}) - \sigma(X_{(t-\delta)_+}) \big) \rd W_s \\ 
&\quad + n^{\alpha-\frac12} \int_{(t-\delta)_+}^t \big( K(t-s) - K(t-\underline{s}) \big) \big( \sigma(X_{(t-\delta)_+})- \sigma(X_t) \big) \rd W_s.
\end{align*}
By Minkowski's integral inequality and It\^o's isometry, one gets 
\begin{align*}
\big\|\widehat{R}_{t}^{n} -\widetilde{R}_{t}^{n} \big\|_{L^2}^2 
&\leq C n^{2\alpha-1} \bigg( \int_0^{(t-\delta) + } \big| K(t-s) - K(t-\underline{s}) \big|^2 \big( \big\|\sigma(\widehat{X}_{\underline{s}}) \big\|_{L^2}^2 + \big\| \sigma(X_{t}) \big\|_{L^2}^2 \big) \rd s \\ 
&\qquad + \int_{(t-\delta)_+ }^t \big| K(t-s) - K(t-\underline{s}) \big|^2 \big\| \sigma(\widehat{X}_{\underline{s}}) - \sigma(X_{(t-\delta)_+}) \big\|_{L^2}^2 \rd s \\
&\qquad + \int_{(t-\delta) + }^t \big| K(t-s) - K(t-\underline{s}) \big|^2 \big\|\sigma(X_{(t-\delta)_+}) -\sigma(X_{t}) \big\|_{L^2}^2 \rd s \bigg) \\ 
&=: C n^{2\alpha-1 } \big( (\mathrm{I}_n) + (\mathrm{II}_n) + (\mathrm{III}_n) \big).
\end{align*}
An application of the linear growth of $\sigma$ and \eqref{eq.momentBounded}, as well as $K(\cdot) = \cK(\cdot) \cE(\cdot)$, \eqref{eq.def:k&e}, \eqref{eq.cU_Holder}, \cite[Lemma 4.2]{DaiXiaoBu2022} and \cite[Lemma 6.3]{NualartSaikia2023} yields 
\begin{align}
(\mathrm{I}_n) 
&\leq \sup_{r \in [0, T]} \big( \big\|\sigma(\widehat{X}_{r}) \big\|_{L^2}^2 + \big\|\sigma(X_{r}) \big\|_{L^2}^2 \big) \int_{0}^{(t-\delta) + } \big| K(t-s) - K(t-\underline{s}) \big|^2 \rd s \nonumber \\ 
&\leq C \int_{0}^{(t-\delta) + } \big|(t-s)^{\alpha} - (t-\underline{s})^{\alpha} \big|^{2} \big|\cK(t-\underline{s})\big|^2 \rd s + C \int_{0}^{(t-\delta) + } \big|\cK(t-s) - \cK(t-\underline{s})\big|^{2} \rd s \nonumber \\ 
&\leq C n^{-2\alpha} + C n^{-2} \delta^{2\alpha-3}. \label{eq.J_1^n, 1}
\end{align} 
For the term $(\mathrm{II}_n)$, using \eqref{eq.lips} and \eqref{X-hatX_Lp} deduces that for any $\delta > n^{-1}$ and $s \in [(t-\delta)_{ + }, t]$, 
\begin{align*}
\big\|\sigma(\widehat{X}_{\underline{s}}) - \sigma(X_{(t-\delta)_{ + }}) \big\|_{L^2}^2 
\leq C \big| \underline{s} - (t-\delta)_{ + } \big|^{2\alpha-1} + C n^{1-2\alpha} 
\leq C \delta^{2\alpha-1} + C n^{1-2\alpha}
 \leq C \delta^{2\alpha-1}, 
\end{align*}
which together with \eqref{eq.K-reg} implies 
\begin{align}
(\mathrm{II}_n) 
\leq C \delta^{2\alpha- 1 } \int_{0}^{t} \big| K(t-s) - K(t-\underline{s}) \big|^2 \rd s 
\leq C n^{1-2\alpha} \delta^{2\alpha-1}. 
\end{align} 
For the term $(\mathrm{III}_n)$, using \eqref{eq.lips}, \eqref{eq.solutionReqularity} and \eqref{eq.K-reg} deduces 
\begin{align}
(\mathrm{III}_n) 
\leq C \delta^{2\alpha - 1} \int_{0}^{t} \big| K(t-s) - K(t-\underline{s}) \big|^2 \rd s \leq C n^{1-2\alpha} \delta^{2\alpha-1}. \label{eq.J_3^n, 1} 
\end{align} 
Thus, it follows from \eqref{eq.J_1^n, 1}--\eqref{eq.J_3^n, 1} that 
\begin{align*}
\limsup_{n \to \infty} \sup_{t \in [0, T]} \big\|\widehat{R}_{t}^{n} - \widetilde{R}_{t}^{n} \big\|_{L^2}^2 
\leq C \delta^{2\alpha-1}, 
\end{align*}
where $C>0$ is independent of $\delta$. Finally, by letting $\delta > 0$ arbitrarily small, the proof of Lemma \ref{lem.replaceR_n1}(ii) is completed. 
\hfill$\Box$

\begin{lemma} \label{lm.cM_limit_Cov} 
Let $(i,i_1,i_2,j,j_1,j_2)\in\{1,2,\cdots,d\}^6$ and $(\ell,\ell_1,\ell_2)\in\{1,2,\cdots,m\}^3$. For $t \in [0,T]$ and integer $n \geq 1$, define
\begin{align} \label{eq.def:cM_n}
 (\cM_n)_{\ell,t}^{i,j} 
= n^{\alpha-\frac12} \int_0^{t} K^{i}_{j} (t-s) - K^{i}_{j} (t -\underline{s}) \rd W_{s}^{\ell}.
\end{align}
Then, for $0 \leq t_1 < t_2 \leq T$, 
\begin{align} \label{eq.cM_limit_t1t2}
\lim_{n\to \infty} \hE \big[ (\cM_n)_{\ell_1,\underline{t_1}}^{i_1,j_1} (\cM_n)_{\ell_2,\underline{t_2}}^{i_2,j_2} \big] = 0. 
\end{align}
Moreover, for $t \in (0,T]$, 
\begin{align} \label{eq.cM_limit_t}
\lim_{n\to \infty} \hE \big[ (\cM_n)_{\ell_1,\underline{t}}^{i_1,j_1} (\cM_n)_{\ell_2,\underline{t}}^{i_2,j_2} \big] 
= 
\begin{cases}
\kappa_2^2(\alpha), &i_1=j_1, i_2=j_2, \text{ and } \ell_1=\ell_2,\\
0,&\text{otherwise}, 
\end{cases}
\end{align}
where the constant $\kappa_2^2(\alpha)$ is defined in \eqref{eq.kappa2}. 
\end{lemma}

\begin{proof}
 For any $0 \leq t_1<t_2 \leq T$, it follows from \eqref{eq.ItoProduct Rule} that 
 {\small 
\begin{align*}
\hE \big[ (\cM_n)_{\ell_1,\underline{t_1}}^{i_1,j_1} (\cM_n)_{\ell_2,\underline{t_2}}^{i_2,j_2} \big]
= n^{2\alpha-1} \hE \left[ \int_0^{\underline{t_1}}\big( K^{i_1}_{j_1} (\underline{t_1}-s) - K^{i_1}_{j_1} (\underline{t_1} -\underline{s}) \big) \big( K^{i_2}_{j_2} (\underline{t_2}-s) - K^{i_2}_{j_2} (\underline{t_2} -\underline{s}) \big) \rd \<W^{\ell_{1}},W^{\ell_{2}}\>_{s}\right].
\end{align*}}
\!\!By $K(\cdot) = \cK(\cdot) \cE(\cdot)$, \eqref{eq.def:k&e}, and the change of variables $y = \lfloor nt_1 \rfloor - ns$, one can arrive at 
\begin{align*}
\big| \hE \big[ (\cM_n)_{\ell_1,\underline{t_1}}^{i_1,j_1} (\cM_n)_{\ell_2,\underline{t_2}}^{i_2,j_2} \big] \big| 
&\leq n^{2\alpha-1} \int_0^{\underline{t_1}} \big| K^{i_1}_{j_1} (\underline{t_1}-s) - K^{i_1}_{j_1} (\underline{t_1} -\underline{s}) \big| \big| K^{i_2}_{j_2} (\underline{t_2}-s) - K^{i_2}_{j_2} (\underline{t_2} -\underline{s}) \big| \rd s \notag \\ 
&= n \int_0^{\underline{t_1}} \big| \cK( \lfloor nt_1 \rfloor-ns) \cE^{i_1}_{j_1}(\underline{t_1} -s) - \cK( \lfloor nt_1 \rfloor - \lfloor ns \rfloor) \cE^{i_1}_{j_1}\left(\underline{t_1} - \underline{s} \right) \big| \notag \\ 
&\qquad\ \ \, \times \big| \cK( \lfloor nt_2 \rfloor-ns) \cE^{i_2}_{j_2}(\underline{t_2} -s) - \cK( \lfloor nt_2 \rfloor - \lfloor ns \rfloor) \cE^{i_2}_{j_2}\left(\underline{t_2} - \underline{s} \right) \big| \rd s \notag \\ 
&= \int_0^{ \lfloor nt_1 \rfloor} g^n(y) \rd y, 
\end{align*}
where
{\small
\begin{align*}
g^n(y) 
&:= \left| \cK(y) \cE^{i_1}_{j_1}\left(\frac{y}{n}\right) - \cK( \lceil y\rceil ) \cE^{i_1}_{j_1} \left(\frac{\lceil y\rceil}{n}\right) \right| \\ 
&\quad\ \times \left| \cK(y + \lfloor nt_2 \rfloor- \lfloor nt_1 \rfloor) \cE^{i_2}_{j_2}\left(\frac{y + \lfloor nt_2 \rfloor- \lfloor nt_1 \rfloor}{n}\right) 
- \cK( \lceil y\rceil + \lfloor nt_2 \rfloor- \lfloor nt_1 \rfloor) \cE^{i_2}_{j_2}\left(\frac{\lceil y\rceil + \lfloor nt_2 \rfloor- \lfloor nt_1 \rfloor}{n}\right) \right|. 
\end{align*}} 
\!\!Then, by a similar argument as in the proof of Lemma \ref{lm.cB_n}, one gets $\lim_{n \to \infty} \big| \hE\big[ (\cM_n)_{\ell_1,\underline{t_1}}^{i_1,j_1} (\cM_n)_{\ell_1,\underline{t_2}}^{i_2,j_2} \big] \big| = 0$, and hence \eqref{eq.cM_limit_t1t2} holds.

Clearly, $\hE \big[ (\cM_n)_{\ell_1,\underline{t}}^{i_1,j_1} (\cM_n)_{\ell_2,\underline{t}}^{i_2,j_2} \big] = 0$ when $\ell_1\neq \ell_2$. Thus, it remains to show that \eqref{eq.cM_limit_t} holds when $\ell_1=\ell_2=\ell$. Using $K(\cdot) = \cK(\cdot) \cE(\cdot)$, \eqref{eq.def:k&e} and the change of variables $ns - q = x$ yields 
{\small 
\begin{align*}
&\ \hE \big[ (\cM_n)_{\ell,\underline{t}}^{i_1,j_1} (\cM_n)_{\ell,\underline{t}}^{i_2,j_2} \big] 
=n^{2\alpha-1} \int_0^{\underline{t}} \big( K^{i_1}_{j_1} (\underline{t}-s) - K^{i_1}_{j_1} (\underline{t} -\underline{s}) \big) \big( K^{i_2}_{j_2} (\underline{t}-s) - K^{i_2}_{j_2} (\underline{t} -\underline{s}) \big) \rd s \\ 
&=n^{2\alpha-1} \sum_{q=0}^{ \lfloor nt \rfloor-1} \int_{\frac qn}^{\frac{q+1}{n}} \left( K^{i_1}_{j_1}(\underline{t}-s)- K^{i_1}_{j_1}\left(\underline{t}-\frac qn\right) \right)\left( K^{i_2}_{j_2}(\underline{t}-s)- K^{i_2}_{j_2}\left(\underline{t}-\frac qn\right) \right) \rd s\\ 
&=\sum_{q=0}^{ \lfloor nt \rfloor-1} \int_{0}^{1} \left( \cK( \lfloor nt \rfloor-q-x) \cE^{i_1}_{j_1}\left(\frac{ \lfloor nt \rfloor-q-x}{n}\right) -\cK( \lfloor nt \rfloor-q) \cE^{i_1}_{j_1}\left(\frac{ \lfloor nt \rfloor -q}{n}\right)\right)\\
&\qquad\qquad\times\left( \cK( \lfloor nt \rfloor-q-x) \cE^{i_2}_{j_2}\left(\frac{ \lfloor nt \rfloor-q-x}{n}\right) -\cK( \lfloor nt \rfloor-q) \cE^{i_2}_{j_2}\left(\frac{ \lfloor nt \rfloor -q}{n}\right)\right) \rd x\\
&=\sum_{m=1}^{ \lfloor nt \rfloor} \int_{0}^{1} \left( (m-x)^{\alpha-1} \cE^{i_1}_{j_1}\left(\frac{m-x}{n}\right) -m^{\alpha-1} \cE^{i_1}_{j_1}\left(\frac{m}{n}\right)\right)\left( (m-x)^{\alpha-1} \cE^{i_2}_{j_2}\left(\frac{m-x}{n}\right) -m^{\alpha-1} \cE^{i_2}_{j_2}\left(\frac{m}{n}\right)\right) \rd x.
\end{align*}} 
\!\!Note that 
{\small 
\begin{align*}
\lim_{n\to\infty} \sum_{m= \lfloor nt \rfloor}^{\infty} \int_{0}^{1} \left| (m-x)^{\alpha-1} \cE\left(\frac{m-x}{n}\right) -m^{\alpha-1} \cE\left(\frac{m}{n}\right)\right|^2 \rd x = 0 
\end{align*} }
\!\!and 
{\small 
\begin{align*}
& \lim_{n\to\infty}
\left( (\lceil y \rceil-x)^{\alpha-1} \cE^{i_1}_{j_1}\left(\frac{\lceil y \rceil-x}{n}\right) -(\lceil y \rceil)^{\alpha-1} \cE^{i_1}_{j_1}\left(\frac{\lceil y \rceil}{n}\right)\right)
\left( (\lceil y \rceil-x)^{\alpha-1} \cE^{i_2}_{j_2}\left(\frac{\lceil y \rceil-x}{n}\right) -(\lceil y \rceil)^{\alpha-1} \cE^{i_2}_{j_2}\left(\frac{\lceil y \rceil}{n}\right)\right)\\
&= 
\begin{cases}
\frac{1}{\Gamma^{2}(\alpha)} \big( (\lceil y \rceil-x)^{\alpha-1} - (\lceil y \rceil)^{\alpha-1}\big)^2, & i_1 = j_1 \mbox{ and } i_2 = j_2,\\
0, & \text{otherwise}, 
\end{cases} \qquad \forall\, m \in N_{+} 
\end{align*} }
\!\!due to \eqref{eq.limeij} and \eqref{eq.cU_Holder}. Then, applying the DCT w.r.t.\ $ \rd y \otimes \rd x$ shows 
{\small 
\begin{align*}
&\ \lim_{n\to\infty} \hE \big[ (\cM_n)_{\ell,\underline{t}}^{i_1,j_1} (\cM_n)_{\ell,\underline{t}}^{i_2,j_2} \big] \\ 
&=\lim_{n\to\infty}\sum_{m=1}^{\infty} \int_{0}^{1} \left( (m-x)^{\alpha-1} \cE^{i_1}_{j_1}\left(\frac{m-x}{n}\right) -m^{\alpha-1} \cE^{i_1}_{j_1}\left(\frac{m}{n}\right)\right)\left( (m-x)^{\alpha-1} \cE^{i_2}_{j_2}\left(\frac{m-x}{n}\right) -m^{\alpha-1} \cE^{i_2}_{j_2}\left(\frac{m}{n}\right)\right) \rd x \\
&=\lim_{n\to\infty}\int_{0}^{1}\int_{0}^{\infty} \left( (\lceil y \rceil-x)^{\alpha-1} \cE^{i_1}_{j_1}\left(\frac{\lceil y \rceil-x}{n}\right) -(\lceil y \rceil)^{\alpha-1} \cE^{i_1}_{j_1}\left(\frac{\lceil y \rceil}{n}\right)\right)\\
&\qquad\qquad \times \left( (\lceil y \rceil-x)^{\alpha-1} \cE^{i_2}_{j_2}\left(\frac{\lceil y \rceil-x}{n}\right) -(\lceil y \rceil)^{\alpha-1} \cE^{i_2}_{j_2}\left(\frac{\lceil y \rceil}{n}\right)\right) \rd y \rd x\\
&=
\begin{cases}
\kappa_2^2(\alpha), & i_1 = j_1 \mbox{ and } i_2 = j_2,\\ 
0, & \text{otherwise}. 
\end{cases}
\end{align*} }
\!\!Therefore, one can also conclude that \eqref{eq.cM_limit_t} holds. 
\end{proof}

\vskip 0.5em 
\textit{Proof of Lemma \ref{lm.R_n1}}. 
We formulate $\widetilde{R}_{t}^{n} = (\widetilde{R}_{t}^{n,1}, \widetilde{R}_{t}^{n,2}, \cdots, \widetilde{R}_{t}^{n,d})^{\top}$ for $t \in [0,T]$. By \eqref{eq.def_tildR_tn1} and \eqref{eq.def:cM_n}, it holds that for any $i \in \{1,2,\cdots,d\}$, 
\begin{align*} 
\widetilde{R}_{t}^{n,i} 
= n^{\alpha - \frac{1}{2}} \sum_{j=1}^{d} \sum_{\ell=1}^{m} \int_0^t \big( K^{i}_{j}(t-s) - K^{i}_{j}(t-\underline{s}) \big) \sigma^j_{\ell}(X_t) \rd W^{\ell}_{s} 
= \sum_{j=1}^{d} \sum_{\ell=1}^{m} \sigma^j_{\ell}(X_t) (\cM_n)_{\ell,t}^{i,j}. 
\end{align*} 
Recall \eqref{eq.def:cM_n}, and let $\cM_n := \{ (\cM_n)_{\ell,\underline{t}}^{i,j}, t \in [0,T] \}$ 
denote a $d^2 m$-dimensional centered Gaussian process. Next, the proof of this lemma is divided into the following two steps.

\underline{Step 1}: In this step, we are devoted to showing that the finite-dimensional distributions of the process $\cM_n$ converge in distribution to those of $\kappa_2(\alpha)Z$. It follows from Lemma \ref{lm.cM_limit_Cov} that the covariance of $\cM_n$ converges to the covariance of $\kappa_2(\alpha)Z$. Thus, we only need to show that the correlation between the processes $\cM_n$ and $W$ converges to zero.

For any $(i,j)\in \{1,2,\cdots,d\}^2$, $(\ell,q)\in\{1,2,\cdots,m\}^2$ and interval $[a, b]$ with $0 \leq a < b \leq T$, using \eqref{eq.ItoProduct Rule} and \eqref{eq.K-reg} shows 
\begin{align*}
\big| \hE \big[ (\cM_n)_{\ell,\underline{t}}^{i,j}\,(W_b^q-W_a^q) \big] \big| 
\leq n^{\alpha-\frac12} \int_{0}^{\underline{t}} \left| K^{i}_{j}(\underline{t}-s)-K^{i}_{j}(\underline{t}-\underline{s})\right| \rd s 
\leq C n^{-\frac12}, 
\end{align*}
which implies $\lim_{n\to\infty} \hE \big[ (\cM_n)_{\ell,\underline{t}}^{i,j}\,(W_b^q-W_a^q) \big] = 0$. The proof of Step 1 is complete.

\underline{Step 2}. The results established in Step 1 imply that, for any square-integrable random variable $G$ that is measurable w.r.t.\ $W$, any $\hR^{d^2m}$-valued $\lambda_{1}, \dots, \lambda_{N}$, and any points $t_{1}, \dots, t_{N} \in (0, T]$, one can claim that 
{\small 
\begin{align} \label{eq.stable cov}
\lim_{n \to \infty} \hE \left[G \exp \left(\bm{i}\sum_{i=1}^{d}\sum_{j=1}^{d} \sum_{\ell=1}^{m}\sum_{q=1}^{N} (\lambda_q)^{i,j}_{\ell} (\cM_n)_{\ell,\underline{t_{q}}}^{i,j} \right) \right] 
= \hE [G] \exp \left( - \frac{\kappa^{2}_{2}(\alpha)}{2}\sum_{i=1}^{d}\sum_{j=1}^{d} \sum_{\ell=1}^{m}\sum_{q=1}^{N} (\lambda_q)^{i,i}_{\ell} (\lambda_q)^{j,j}_{\ell} \right).
\end{align} }
\!\!Indeed, when $G = c$ with $c$ a deterministic constant, \eqref{eq.stable cov} holds by Lemma \ref{lm.cM_limit_Cov}. For the other case, by an approximation argument in $L^2(\Omega)$, we can assume that $G = \exp \left( \sum_{q_1=1}^L\mu_{q_1}^{\top}(W_{s_{q_1}}-W_{s_{q_{1}-1}}) \right)$, where $\mu_1, \ldots, \mu_L \in\hR^d$ and $0 = s_0<s_1<\cdots<$ $s_L = T$, and hence 
\begin{align*}
&\ \lim_{n \to \infty}E \left[ \exp \left( \sum_{i=1}^d \sum_{q_1=1}^L\mu_{q_1}^i(W_{s_{q_1}}^i-W_{s_{q_1-1}}^i) + \bm{i}\sum_{i=1}^{d}\sum_{j=1}^{d} \sum_{\ell=1}^{m}\sum_{q=1}^{N} (\lambda_{q})^{i,j}_{\ell} (\cM_n)_{\ell,\underline{t_q}}^{i,j} \right) \right] \\
&= \exp \left(\frac12 \sum_{i=1}^d\sum_{q_1=1}^L(\mu_{q_1}^i)^2(s_{q_1}-s_{q_1-1}) - \frac{\kappa^{2}_{2}(\alpha)}{2}\sum_{i=1}^{d}\sum_{j=1}^{d} \sum_{\ell=1}^{m}\sum_{q=1}^{N} (\lambda_q)^{i,i}_{\ell} (\lambda_q)^{j,j}_{\ell} \right) \\
&= \hE [G] \exp \left( - \frac{\kappa^{2}_{2}(\alpha)}{2}\sum_{i=1}^{d}\sum_{j=1}^{d} \sum_{\ell=1}^{m}\sum_{q=1}^{N} (\lambda_q)^{i,i}_{\ell} (\lambda_q)^{j,j}_{\ell} \right).
\end{align*} 
In particular, by taking $G = \exp\big( \bm{i} \sum_{\ell=1}^m \sum_{j=1}^d\sum_{q_1=1}^L (\rho_{q_1})_{\ell}^{j}\sigma_{\ell}^{j}(X_{s_{q_1}}) \big)$ in \eqref{eq.stable cov}, with $\rho_1, \ldots, \rho_L\in\hR^{dm}$ and $0 = s_0 < s_1 < \cdots < s_L = T$, and applying L\'evy's continuous theorem, we obtain 
\small{\begin{align*}
\left(\{\sigma_{\ell}^{j}(X_{s_{1}})\},\cdots, \{\sigma_{\ell}^{j}(X_{s_L})\},\cM_{n,t_1},\cdots,\cM_{n,t_N}\right)\xrightarrow[n \to \infty]{d}\left(\{\sigma_{\ell}^{j}(X_{s_{1}})\},\cdots, \{\sigma_{\ell}^{j}(X_{s_L})\}, \kappa_2(\alpha)Z_{t_1},\cdots,\kappa_2(\alpha)Z_{t_N}\right) 
\end{align*}}
\!\!in $\hR^{d m L \times d^2 m N}$. Thus, it follows from the continuous mapping theorem that the convergence of the finite-dimensional distributions of $\{\sigma_{\ell}^{j}(X_t) (\cM_n)_{\ell,\underline{t}}^{i,j},t \in (0,T]\}$ to those of $\{\kappa_{2}(\alpha) \sigma_{\ell}^{j}(X_t) Z_{\ell,t}^{i,j},t\in (0,T]\}$, and thereby the finite-dimensional distributions of the process $\{\widetilde{R}^{n}_{\underline{t}},t\in (0,T]\}$ converge to those of $\{ \widetilde{\cR}_t,t\in (0,T]\}$, as $n$ goes to infinity. The proof is completed. 
\hfill$\Box$

\appendix

\section{Singular kernel and related estimates}
 \label{appen:ML function}

Firstly, we introduce the Mittag--Leffler function and its basic properties, see \cite[Chapter 1.2]{Podlubny1999} for more details. Fix $a \in (0, 1)$ and $b \in \hR$. The Mittag--Leffler function is defined by 
\begin{align*}
E_{a, b}(z) = \sum_{k = 0}^{\infty} \frac{z^k}{\Gamma(a k + b)}, \qquad \mbox{for } z \in \hC, 
\end{align*}
where $\Gamma(z) = \int_{0}^{\infty} t^{z-1} e^{-t} \rd t$ denotes the Gamma function. For simplicity, denote $E_{a} (z) = E_{a, 1} (z)$. For any real number $c \in (a\pi/2, a\pi)$, there exists some constant $C = C(a, b, c) > 0$ such that
\begin{align} \label{eq.ML-bound}
| E_{a, b}(z) | \leq C (1 + |z|)^{-1}, \qquad c \leq |\arg(z)| \leq \pi.
\end{align}
In addition, for any $\eta \in [0, 1]$ and $\lambda > 0$, 
\begin{align} \label{eq.MLder}
\frac{ \rd}{ \rd t} [ t^{a + \eta - 1} E_{a, a + \eta}(- \lambda t^{a}) ]
= 
\begin{cases}
- \lambda t^{a-1} E_{a, a}(- \lambda t^{a}), &\mbox{if } a + \eta = 1, \\ 
t^{a + \eta-2} E_{a, a + \eta-1}(- \lambda t^{a}), &\mbox{if } a + \eta \neq 1.
\end{cases}
\end{align}

For convenience, for $\alpha \in (\frac{1}{2},1)$ and the negative definite matrix $A \in \hR^{d \times d}$, define the functions 
\begin{align} \label{eq.def:k&e}
\cK(u) = u^{\alpha-1} \qquad \mbox{and} \qquad 
 \cE(u) = E_{\alpha, \alpha} (A u^{\alpha}), \qquad 
\mbox{for}~~u \in (0,\infty).
\end{align}
Then the kernel $K(u) = u^{\alpha-1} E_{\alpha, \alpha} (A u^{\alpha}) = \cK(u) \cE(u)$. For the function $\cK(\cdot)$, it follows from \cite[Lemmas 4.1 and 4.2]{DaiXiaoBu2022} that for any $t \in [0,T]$, 
\begin{align}
\begin{aligned}
\int_{0}^{t} |\cK(t - s) - \cK(\underline{t} - \underline{s})| \rd s \leq C h^{\alpha}, \qquad
\int_{0}^{t} |\cK(t - s) - \cK(\underline{t} - \underline{s})|^2 \rd s \leq C h^{2\alpha-1},\\
\int_{0}^{\underline{t}} |\cK(t - s) - \cK(\underline{t} - \underline{s})| \rd s \leq C h^{\alpha}, \qquad
\int_{0}^{\underline{t}} |\cK(t - s) - \cK(\underline{t} - \underline{s})|^2 \rd s \leq C h^{2\alpha-1}. 
\end{aligned} \label{eq.k-reg}
\end{align}
In view of \eqref{eq.ML-bound} and \eqref{eq.MLder}, $ \cE(\cdot)$ is a bounded $\alpha$-H\"older continuous function, that is to say, there exists $C > 0$ such that
\begin{align} \label{eq.cU_Holder} 
| \cE(u) | \leq C \qquad \mbox{and} \qquad | \cE(u_1) - \cE(u_2) | \leq C |u_1 - u_2|^{\alpha}, \qquad \forall\, u, u_1, u_2 \in (0,T], 
\end{align}
which implies that for any $t \in [0,T]$, 
\begin{align} \label{eq.|K|}
\int_{t}^{t+h} |K(s)| \rd s \leq C h^{\alpha}, \qquad 
\int_{t}^{t+h} |K(s)|^2 \rd s \leq C h^{2\alpha - 1}.
\end{align}
Moreover, recalling \eqref{eq.k-reg} \label{eq.u-Lip} yields that for any $t \in [0,T]$, 
\begin{equation}
\begin{aligned}
\int_{0}^{t} |K(t - s) - K(t - \underline{s})| \rd s \leq C h^{\alpha}, \qquad
\int_{0}^{t} |K(t - s) - K(t - \underline{s})|^2 \rd s \leq C h^{2\alpha-1},\\
\int_{0}^{\underline{t}} |K(t - s) - K(\underline{t} - \underline{s})| \rd s \leq C h^{\alpha}, \qquad
\int_{0}^{\underline{t}} |K(t - s) - K(\underline{t} - \underline{s})|^2 \rd s \leq C h^{2\alpha-1}. 
\end{aligned}
 \label{eq.K-reg}
\end{equation}

\begin{lemma} \label{lm.KHZ}
Let $t \in (0, T]$ and $\beta^{*} = \beta / (\beta-2)$ with $\beta \in (2,1/(1-\alpha))$. Then for any adapted $\hR^{d}$-valued process $H$ and $\hR^{d \times m}$-valued process $Z$, it holds that 
\begin{align}
& \hE \left[ \left| \int_{0}^{t} K(t-s) H_{s} \rd s \right|^{p} \right] \leq C \int_{0}^{t} \hE \left[ \left|H_{s} \right|^{p} \right] \rd s, \qquad\ \ \forall\, p \geq 2, \label{lm.KH} \\
& \hE \left[ \left| \int_{0}^{t} K(t-s) Z_{s} \rd W_{s} \right|^{p} \right] \leq C \int_{0}^{t} \hE \left[ \left|Z_{s} \right|^{p} \right] \rd s, \qquad \forall\, p > 2 \beta^{*}, \label{lm.KZ} \\
& \hE \left[ \left| \int_{0}^{\underline{t}} ( K(t-s) - K(\underline{t}-\underline{s})) H_{s} \rd s \right|^{p} \right] + \hE \left[ \left| \int_{\underline{t}}^{t} K(t-s) H_{s} \rd s \right|^{p} \right] \notag \\ 
&\qquad \leq C h^{\alpha p} \sup_{r \in [0, T]} \hE \left[ \left|H_{r} \right|^{p} \right], \qquad\quad\ \ \forall\, p \geq 1, \label{lm.|K-K|H} \\
&\hE \left[ \left| \int_{0}^{\underline{t}} ( K(t-s) - K(\underline{t}-\underline{s})) Z_{s} \rd W_{s} \right|^{p} \right] + \hE \left[ \left| \int_{\underline{t}}^{t} K(t-s) Z_{s} \rd W_{s} \right|^{p} \right] \notag \\ 
&\qquad \leq C h^{(\alpha -\frac{1}{2})p} \sup_{r \in [0, T]} \hE \left[ \left|Z_{r} \right|^{p} \right], \qquad \forall\, p \geq 2, \label{lm.|K-K|Z}
\end{align}
where $C>0$ depends only on $\alpha$, $\beta$, $p$, $A$ and $T$.
\end{lemma}

\begin{proof}
Let $p \geq 2$, using H\"older's inequality and \eqref{eq.cU_Holder} indicates 
\begin{align*}
\hE \left[ \left| \int_{0}^{t} K(t-s) H_{s} \rd s \right|^{p} \right] 
\leq C \left( \int_0^t |K(t-s)|^{\frac{p}{p-1}} \rd s \right)^{p-1} \int_0^t \hE \left[|H_{s}|^{p} \right] \rd s 
\leq C \int_{0}^{t} \hE \left[|H_{s}|^{p} \right] \rd s, 
\end{align*}
which implies that \eqref{lm.KH} holds.

Let $p > 2\beta^*$, using the BDG inequality, H\"older's inequality and Fubini's theorem yields
\begin{align*}
\hE \left[ \left| \int_{0}^{t} K(t-s) Z_{s} \rd W_{s} \right|^{p} \right] 
&\leq C \hE \left[ \left( \int_0^t | K(t-s) |^2|Z_s|^2 \rd s \right)^{\frac{p}{2}} \right] \\ 
&\leq C \left( \int_0^t (t-s)^{2(\alpha-1)\frac{\beta}{2}} \rd s \right)^{\frac{2}{\beta} \cdot \frac{p}{2}} \hE \left[ \left( \int_0^t |Z_s|^{2\beta^*} \rd s \right)^{\frac{1}{\beta^*} \cdot \frac{p}{2}} \right] \\ 
&\leq C \left( \int_0^t (t-s)^{\beta(\alpha-1)} \rd s \right)^{\frac{p}{\beta}} \hE \left[t^{\frac{p}{2\beta^*}-1} \int_0^t |Z_s|^p \rd s \right] \\
& \leq C \int_0^t \hE [|Z_s|^p] \rd s, 
\end{align*}
which implies that \eqref{lm.KZ} holds.

Let $p \geq 1$, using Minkowski's integral inequality, \eqref{eq.|K|} and \eqref{eq.K-reg} reads 
\begin{align*}
&\ \hE \left[ \left| \int_{0}^{\underline{t}} ( K(t-s) - K(\underline{t}-\underline{s})) H_{s} \rd s \right|^{p} \right] + \hE \left[ \left| \int_{\underline{t}}^{t} K(t-s) H_{s} \rd s \right|^{p} \right] \\ 
&\leq \left( \int_{0}^{\underline{t}} | K(t-s) - K(\underline{t}-\underline{s})| \hE \left[|H_{s}|^{p} \right]^{\frac{1}{p}} \rd s \right)^{p} + \left( \int_{\underline{t}}^{t} | K(t-s) | \hE \left[|H_{s}|^{p} \right]^{\frac{1}{p}} \rd s \right)^{p} \\ 
&\leq \left[ \left( \int_{0}^{\underline{t}} \left| K(t-s) - K(\underline{t}-\underline{s}) \right| \rd s \right)^{p} + \left( \int_{\underline{t}}^{t} | K(t-s) | \rd s \right)^{p} \right] \sup_{r \in [0, T]} \hE \left[|H_{r}|^{p} \right] \\ 
&\leq C h^{\alpha p} \sup_{r \in [0, T]} \hE \left[|H_{r}|^{p} \right], 
\end{align*}	
which implies that \eqref{lm.|K-K|H} holds.

Let $p \geq 2$, using the BDG inequality, Minkowski's integral inequality, \eqref{eq.|K|} and \eqref{eq.K-reg} shows 
\begin{align*}
&\ \hE \left[ \left| \int_{0}^{\underline{t}} ( K(t-s) - K(\underline{t}-\underline{s})) Z_{s} \rd W_{s} \right|^{p} \right] + \hE \left[ \left| \int_{\underline{t}}^{t} K(t-s) Z_{s} \rd W_{s} \right|^{p} \right] \\ 
&\leq C \hE \left[ \left( \int_0^{\underline{t}} | K(t-s) - K(\underline{t}-\underline{s})|^2|Z_s|^2 \rd s \right)^{\frac{p}{2}} \right] + C \hE \left[ \left( \int_{\underline{t}}^{t}| K(t-s) |^2|Z_s|^2 \rd s \right)^{\frac{p}{2}} \right] \\ 
&\leq C \left[ \left( \int_0^{\underline{t}} | K(t-s) - K(\underline{t}-\underline{s})|^2 \rd s \right)^{\frac{p}{2}} + \left( \int_{\underline{t}}^{t}| K(t-s) |^2 \rd s \right)^{\frac{p}{2}} \right] \sup_{r \in [0, T]} \hE \left[|Z_{r}|^{p} \right] \\ 
&\leq C h^{(\alpha-\frac{1}{2})p} \sup_{r \in [0, T]} \hE \left[|Z_{r}|^{p} \right], 
\end{align*}
which implies that \eqref{lm.|K-K|Z} holds. The proof is completed.
\end{proof}


\end{document}